\documentclass[reqno]{amsart}
\usepackage[utf8]{inputenc}
\usepackage[T1]{fontenc}
\usepackage[english,activeacute]{babel}
\usepackage{mathtools}
\usepackage{amssymb}
\usepackage{amsfonts}
\usepackage{amsthm}
\usepackage{calrsfs}
\usepackage{latexsym}
\usepackage{array}
\usepackage{bbm}
\usepackage{stmaryrd}
\usepackage{url}
\usepackage[svgnames]{xcolor}
\usepackage{hyperref}
\usepackage{bm}
\newcommand{\m}[1]{\mathbb{#1}}

\newcommand{\q}[1]{\mathcal{#1}}
\DeclareMathOperator{\Span}{\mathrm{Span}}
\renewcommand{\le}{\leqslant}
\renewcommand{\ge}{\geqslant}
\theoremstyle{plain}
\newtheorem{thm}{Theorem}
\newtheorem*{thm*}{Theorem}
\newtheorem{prop}[thm]{Proposition}

\newtheorem{lem}{Lemma}

\theoremstyle{definition}

\theoremstyle{remark}
\newtheorem{nb}{Remark}
\newtheorem*{claim}{Claim}
\newcommand{\R}{\mathbb{R}}
\newcommand{\N}{\mathbb{N}}

\newcommand{\ds}{\displaystyle}
\newcommand{\opH}{\mathbf{H}}

\newcommand{\opB}{\mathbf{B}}
\newcommand{\opL}{\mathbf{L}}
\newcommand{\opJ}{\mathbf{J}}
\newcommand{\br}{r}
\newcommand{\bpsi}{\psi^0}
\newcommand{\brho}{\rho}
\newcommand{\bg}{g}

\title{Multi-travelling waves for the nonlinear Klein-Gordon equation}
\author{Rapha\"el C\^ote and Yvan Martel}

\thanks{R.~C. and Y.~M. were supported in part by the ERC advanced grant 291214 BLOWDISOL. R.~C. was also supported in part by the ANR contract MAToS ANR-14-CE25-0009-01.}
\thanks{The authors thanks Xu Yuan (\'Ecole polytechnique) for pointing out an error in~\cite{CoteMar}.}

\begin{document} 
\allowdisplaybreaks

\begin{abstract}
For the nonlinear Klein-Gordon equation in $\R^{1+d}$, we prove the existence of multi-solitary waves  made of any number $N$ of decoupled bound states. This extends the work of C\^ote and Mu\~noz  \cite{CoMu} (Forum Math. Sigma {\bfseries2} (2014)) which was restricted to ground states, as were most previous similar results  for other nonlinear dispersive and wave models.
\end{abstract}

\maketitle

\section{Introduction}
In this paper we extend previous constructions of multi-solitary wave solutions for the nonlinear Klein-Gordon equation (NLKG) in $\R^{1+d}$, $d \ge 1$,
\begin{equation*}
\label{NLKG}
\tag{NLKG}
\partial_{tt} u - \Delta u + u -f(u)  = 0, \quad u(t,x) \in \R, \quad (t,x)\in \R \times \R^d. 
\end{equation*}
This equation arises in Quantum Field Physics as a model for a self-interacting, nonlinear scalar field, invariant under Lorentz transformations (see below). 
We focus on the particular case where
\begin{equation}\label{d.f}
f(u) = |u|^{p-1} u \quad \hbox{for } 1<p<\frac {d+2}{d-2} \quad \hbox{($p>1$ for $d=1$ or $2$)}
\end{equation}
but the arguments can extended to more general situations.
We set 
\[F(u) = \int_0^u f(v) dv = \frac{|u|^{p+1}}{p+1}.\]
As usual, we   see the (NLKG) equation as a first order system of equations
\begin{align} \label{NLKGs}
\partial_t \begin{pmatrix} u \\\partial_t  u\end{pmatrix}
= \begin{pmatrix}  \partial_t u\\ \Delta u -u+ f(u)\end{pmatrix}.
\end{align}
In this framework, we work with vector data $U=(u, \partial_t u)^{\top}$. We use upper-case letters to denote vector valued functions and lower-case letters for scalar functions.

\smallskip

Recall that the corresponding Cauchy problem for \eqref{NLKG} is locally well-posed in $H^s(\R^d) \times H^{s-1}(\R^d)$, for any $s \ge 1$: we refer to Ginibre-Velo \cite{GV85} and Nakamura-Ozawa \cite{NO01} (when $d=2$) for more details.  

Also under the above conditions, the Energy and Momentum (every integral is taken over $\R^d$)
\begin{align}
 \mathcal E[u,u_t] (t) & =  \frac{1}{2}  \int \left[ |\partial_t u(t,x)|^2+ | \nabla u(t,x)|^2 + |u(t,x)|^2 - 2 F(u(t,x))  \right] dx, \label{E} \\
 \mathcal P[u,u_t]  (t) & =  \frac{1}{2} \int \partial_t u(t,x)  \nabla u(t,x) \ dx,\label{P}
\end{align}
are conserved along the flow. In this paper, we will work in the energy space $H^1(\R^d) \times L^2(\R^d)$ endowed with the following scalar product: denote $U =
(u_1,u_2)^\top$, $V = (v_1,v_2)^\top$, and define
\begin{equation}
\label{nor}
\langle U , V \rangle :=  ( u_1 , v_1 ) + (u_2 , v_2 )\quad
\text{where} \quad (u, v) := \int u v dx.
\end{equation}
We will refer to the orthogonality with respect to $\langle \cdot, \cdot \rangle$ as $L^2$-orthogonality (for vector-valued functions).
We also define the energy norm
\begin{align}\label{Nor}
\|U\|^2 := \langle U , U \rangle + (\nabla u_1,  \nabla u_1) =  \| u_1 \|_{H^1}^2 + \| u_2 \|_{L^2}^2.
\end{align}

Looking for stationary solutions $u(t,x)=q(x)$ of \eqref{NLKG} in $H^1(\R^d)$   
we reduce to the  elliptic PDE
\begin{align}\label{q}
- \Delta q + q - f(q) = 0, \quad   q\in H^1(\R^d).
\end{align}
Let us recall well-known results for equation \eqref{q} from \cite{Ca} (see also references therein).
We call  the solutions of \eqref{q} \emph{bound states}; the set of bound states is denoted by $\mathcal B$
\[
\mathcal B = \{ q :  q \text{ is a nontrivial solution of } \eqref{q}\}.
\]
Standard elliptic arguments (see e.g. \cite{GT77} or Theorem~8.1.1 in \cite{Ca}) show that if $q\in \mathcal B$, then $q$ is of class $C^2(\R^d)$ and has exponential decay as $|x|\to +\infty$, as well as its first and second order derivatives.

Let
\[
\mathcal W(u) = \frac 1 2 \int (|\nabla u|^2 + |u|^2 - 2F(u) ) dx.
\]
We call \emph{ground states} the solutions of \eqref{q} that minimize the functional $W$; the set of ground states is denoted by $\mathcal G$
\[
\mathcal G = \{ q_{\rm GS} : q_{\rm GS} \in \mathcal B \text{ and } \mathcal W(q_{\rm GS}) \le \mathcal W(q) \text{ for all } q\in \mathcal B \}.
\]

Ground states  are now well-understood. In particular, it is well-known (Berestycki-Lions \cite{BL}, Gidas-Ni-Nirenberg \cite{GNN}, Kwong \cite{K}, Serrin-Tang \cite{ST}) that there exists a radial positive function $q_{0}$ of class  $C^2$, exponentially decreasing, along with its first and second derivatives, such that
\[ \mathcal G = \{ q_0 (x-x_0) : x_0\in \m R^d\}. \]

In dimension $1$, it is well-known (by ODE arguments) that $\mathcal B = \mathcal G$.
In contrast, for any $d \ge 2$, it is known that $\mathcal G \subsetneq \mathcal B$: see Remark~8.1.16 in \cite{Ca}, we also refer to Ding \cite{D}, where it is proven that $\q B$ (up to translation) is infinite. Functions $q\in \mathcal B\setminus \mathcal G$ are referred to as \emph{excited states}. Few papers in the literature deal with excited states. Here are some references on the construction of such solutions. Berestycki-Lions \cite{BL2} showed the existence of infinitely many radial \emph{nodal} (i.e sign changing) solutions (see also \cite{HV94} and the references therein). More recently, Del Pino, Musso, Pacard and Pistoia constructed in \cite{DMPP11} solutions to the massless version of  equation \eqref{q} with a centered soliton crowned with negative spikes (rescaled solitons) at the vertices of a regular polygon of radius 1; in \cite{DMPP13}, they constructed sign changing, non radial solutions to \eqref{q} on the sphere $\m S^d$ ($d \ge 4$) whose energy is concentrated along special submanifolds of $\m S^d$.

%{\color{red} Paragraphe a ecrire}

%\bigskip

The main difficulty in dealing with excited states in the evolution equation (NLKG) is the lack of information on the linearized operator $-\Delta z + z  - f'(q)z $. Whereas for ground states, it is known that the linearized operator  has a unique simple negative eigenvalue, and a (non degenerate) kernel given by $\Span( \partial_{x_j} q ; \, j=1,\ldots,d )$, the detailed spectral properties of the linearized operator around general bound states are not known. See Section~2 of this paper.

\smallskip

Since \eqref{NLKG} is invariant under \emph{Lorentz boosts}, given a bound state $q$, we can define its \emph{boosted} counterpart, with relative velocity  $\beta = (\beta_1, \dots, \beta_d) \in \R^d$, where $|\beta|<1$ (we denote by $|\cdot|$ the euclidian norm on $\m R^d$) as
\begin{equation}
\label{Qb}
q_\beta(x) := q(\Lambda_\beta(x)),\quad 
\Lambda_\beta(x) :=   x + (\gamma-1) \frac {\beta (\beta\cdot x)}{|\beta|^2} ,\quad
\gamma:=\frac 1{\sqrt{1-|\beta|^2}}.
\end{equation}
Note that the function $q_\beta$ satisfies
\[
-(\Delta-(\beta\cdot\nabla)^2) q_\beta +q_\beta -f(q_\beta) =0.
\]
In particular,  
\[
R(t,x) = \begin{pmatrix} q_\beta(x-\beta t)\\ 
-\beta \cdot \nabla q_\beta (x-\beta t)\end{pmatrix}
\]
is solution of the (first order system form of the) Klein-Gordon equation \eqref{NLKGs}.

\smallskip

It is well known (see e.g. Grillakis-Shatah-Strauss \cite{GSS1}) that the ground state $(q_0,0)$ is unstable in the energy space (this result was known in the Physics literature as the Derrick's Theorem \cite{D}).
For recent works on the instability properties of $q_0$ and on general solutions with energy slightly above $\q E[(q_0,0)]$, we refer to Nakanishi-Schlag \cite{NS1,NS2} and subsequent works. We also refer to Duyckaerts-Merle \cite{DM1}, in the context of the energy critical nonlinear wave equation for related works.

\smallskip

In this paper, we continue the study of the dynamics of large, quantized energy solutions. Specifically, we  deal with  solutions describing multi-bound states for (\ref{NLKG}),
i.e.  solutions $u$ to \eqref{NLKG} defined on a semi-infinite interval of time, such that
\[ u(t,x) \sim \sum_{n=1}^N q_{n,\beta_n}(x-\beta_n t) \quad \text{as} \quad t \to + \infty, \]
for given speeds $\beta_n$ (all distinct).
Such solutions were constructed   in the context of the nonlinear Schr\"odinger equations, the generalized Korteweg-de Vries equations, the Hartree equation, the energy critical wave equation and the water wave system, by Merle \cite{Merle1}, Martel \cite{Martel1}, Martel-Merle \cite{MM06}, C\^ote-Martel-Merle \cite{CMM}, Combet~\cite{Co1,Co2}, Krieger-Martel-Rapha\"el~\cite{KMR09}, Martel-Merle \cite{MM16} and Ming-Rousset-Tzvetkov \cite{MRT15}, both in stable and unstable contexts (see also references in these works).
For  \eqref{NLKG}, the same result was proved by  Côte-Mu\~noz \cite{CoMu}:  there exist  multi-solitary waves based on the ground state, for the whole range of parameters $\beta_1, \dots, \beta_N \in \m R^d$ (two by two distinct), with $|\beta_n|<1$.

\smallskip

We point out that the above results all concern ground states $q \in \q G$, and rely on the complete description of the linearized operator around the ground state in these cases. To our knowledge, the only work related to excited states is by Côte-Le Coz \cite{CoCo}, for the nonlinear Schr\"odinger equation. In this work, the lack of information on the linearized operator is counterbalanced by assuming that solitary waves are well-separated (high-speed assumption).

\smallskip

The main goal of this paper is to extend the construction of multi-solitary waves to  bound state $q\in \mathcal B$ of the (NLKG) equation, \emph{without assumption on the speeds} (besides their being distinct), thus completing Theorem~1 in~\cite{CoMu}, and opening the way to treat such questions for other  models.

\begin{thm}\label{th:1} Let $N\in\m N\setminus\{0\}$, and $\beta_1, \beta_2 ,\ldots ,\beta_N\in \m R^d$ be such that
\[ \forall n, \quad|\beta_n| <1 \quad \hbox{and}\quad  \forall n' \ne n, \quad \beta_{n'} \ne \beta_n. \]
Let $q_1,q_2, \dots ,q_N \in \q B$ be any bound state solution of equation \eqref{q}.

Then there exist  $T_0>0$, $\omega>0$  and a solution $U$ of \eqref{NLKGs}
in the energy space,   defined for $t\ge T_0$, satisfying
\[
\forall t \ge T_0, \quad  \left\| U(t) - \sum_{n=1}^N R_n(t)
\right\| \le e^{-\omega t}
\]
where
\[
R_n(t,x)=
 \begin{pmatrix} q_{n,\beta_n}(x-\beta_n t)
 \\  -\beta_n\cdot \nabla q_{n,\beta_n}(x-\beta_n t)\end{pmatrix}.
\]
\end{thm}

Using the techniques of this paper, it is possible to extend the main result to more general $H^1$ subcritical nonlinearities.
See e.g. \cite{CoMu} for standard conditions on the nonlinearity.

\smallskip

Recall that for integrable models, like the (KdV) and (mKdV) equations, or the 1D cubic nonlinear Schr\"odinger equation, multi-solitons are explicitely derived from the inverse scattering method. Such solutions are quite special since they are global multi-solitons, both for $t\to \pm \infty$ and describe   elastic collisions of solitons. See e.g. the classical references \cite{KZ,Miura,ZS}.
For nonintegrable equations, in general, the asymptotic behavior as $t\to -\infty$ of multi-bound states as constructed in Theorem~\ref{th:1} is not known.

Recall also that the importance of multi-solitons among all solutions is clearly established by the so-called \emph{soliton resolution conjecture}, which says roughly speaking that any solution of a nonlinear dispersive equation should decomposes in large time as a the sum of certain number of solitons and a dispersive part. See e.g. \cite{Schuur} for a proof in the case of the KdV equation.
We refer to recent works of Duyckaerts, Kenig and Merle \cite{DKM4,DJKM16}, and references therein for general soliton decomposition results in the nonintegrable situation of the energy critical wave equation.

\smallskip

The scheme of the proof is the same as for previous related results, notably \cite{Martel1,CMM,CoMu}: Theorem~\ref{th:1} can be reduced to the existence of solutions to (NLKG) satisfying uniform estimates, that is the following proposition.

\begin{prop}\label{pr:3}
There exist $T_0>0$  and $\omega_0>0$ such that for any $S_0\ge T_0$ there exists $U_0$ such that the solution $U(t)$ of \eqref{NLKGs} with data $U(S_0)=U_0$ is defined in the energy space on the time interval $[T_0,S_0]$ and satisfies
\begin{equation}\label{e:u1}
\forall t\in [T_0,S_0],\quad \left\| U(t) - \sum_{n=1}^N R_n(t)\right\| \le  e^{-\omega_0 t}.
\end{equation}
\end{prop}

Indeed, let $S_m \to +\infty$, and assuming that Proposition \ref{pr:3} holds, let $U_m$ be one solution to (NLKG) satisfying the uniform estimates \eqref{e:u1} on the time interval $[T_0,S_m]$. Using the compactness arguments of Section~4 of \cite{CoMu}, one observes that $(U_m(T_0))_{m \in \m N}$ has a weak-$H^1 \times L^2$ limit $U_0^*$. Then consider the solution $U^*$ to (NLKG) with data $U_0^*$ at time $T_0$: the key feature is that the flow of (NLKG) is continuous for the weak-$H^1 \times L^2$ topology, and this allows to conclude that $U^*$ is the desired multi-soliton. We refer to  \cite[Section 4]{CoMu} for further details.

\smallskip

We are therefore left with solely the proof of Proposition \ref{pr:3}, to which the remainder of this paper is devoted.
To prove it for any bound state, we use two new points:
\begin{enumerate}
\item  a general coercivity argument with no a priori knowledge of the spectral properties of the linearized operator (see Section~2);
\item a simplification of the existence proof so as to deal with possibly multiple degenerate directions, not related to translation invariance (see Section~3).
\end{enumerate}

\section{Spectral theory for bound states}

We consider a bound state $q\in \mathcal B$, a velocity $\beta\in \R^d$, $|\beta|<1$, and the corresponding Lorentz state $q_\beta$ defined by \eqref{Qb}. In this section we are interested in the linearized flow around the solution $R(t,x)$ of \eqref{NLKGs}
\[ R(t,x) = \begin{pmatrix} 
q_\beta(x-\beta t) \\
-\beta \cdot \nabla q_\beta(x-\beta t)
\end{pmatrix}. \]
Define the matrix operator
\[  \opH = \begin{pmatrix} - \Delta + 1 - f'(q_\beta) & - \beta \cdot \nabla \\
  \beta\cdot \nabla & 1 
\end{pmatrix} \quad\hbox{and}\quad \opJ = \begin{pmatrix} 0 & 1 \\
- 1 & 0 
\end{pmatrix}. \]
The (NLKG) equation around $R$, i.e. for solutions of the form 
\[ U(t,x) = R(t,x) + V (t,x-\beta t )\]
where $V$ is a  small perturbation, rewrites as
\begin{equation}\label{eq.V}
 \partial_t V = \opJ\opH V + \q N(V) ,
\end{equation}
where $\q N(V)$ denotes nonlinear terms in $V$.

\subsection{Spectral analysis of $\opJ\opH$}

First, following \cite{CoMu}, we study  the spectral properties of the operator $\opJ\opH$ appearing in 
equation \eqref{eq.V}
\[ 
 \opJ\opH = \begin{pmatrix}
 \beta \cdot \nabla & 1   
 \\ \Delta - 1 + f'(q_\beta) & \beta\cdot \nabla 
\end{pmatrix} \]
 in terms of the spectral properties of  the elliptic operator
\[ \opL =  - \Delta + 1 - f'(q). \]
%\[ \tilde \opL = - \Delta + (\beta \cdot \nabla)^2 + 1 - f'(q_\beta) \]

\begin{lem}\label{lem:L_JH}
{\rm (i)} \emph{Spectral properties of $\opL$.} The self-adjoint operator 
$\opL$ has essential spectrum $[1, +\infty)$, a finite number $\bar k\ge 1$ of negative eigenvalues (counted with multiplicity) and   its kernel is of finite dimension $\bar \ell\ge d$. Let $(\phi_k)_{k=1, \dots, \bar k}$ be an $L^2$-orthogonal family of eigenfunctions of $\opL$ with negative eigenvalues $(-\lambda_k^2)_{k=1, \dots, \bar k}$, and $(\phi^0_\ell)_{\ell = 1, \dots, \bar \ell}$ be an $L^2$-orthogonal basis of $\ker(\opL)$, i.e.
\begin{align} 
\opL \phi_k &= - \lambda_k^2 \phi_k, \quad \lambda_k >0, & k =1, \dots, \bar k
\\
\opL \phi^0_\ell &= 0, &  \ell = 1, \dots, \bar \ell.
\end{align}
Then, there exists $c >0$ such that for any $v \in H^1$ satisfying $(v,\phi_k) = (v,\phi^\ell_0) =0$ for all $k =1, \dots, \bar k$, $\ell = 1, \dots, \bar \ell$, the following holds
\begin{align} \label{eq:coer_L}
(\opL v,v) \ge c \| v \|_{H^1}^2.
\end{align}
{\rm (ii)} \emph{Spectral properties of $\opJ\opH $.}
For $k=1, \dots, \bar k$ and $\ell =1 , \dots, \bar \ell$ and signum $\pm$, let
\begin{align}
Y_k^\pm(x) & = e^{\mp \gamma \lambda_k \beta \cdot x} \begin{pmatrix}
\phi_k \\
- \gamma \beta \cdot \nabla \phi_k \pm \gamma \lambda_k \phi_k \\
\end{pmatrix} (\Lambda_\beta(x)), \label{d.Yk}\\
 \Phi^0_\ell(x) & = \begin{pmatrix}
\phi^0_\ell \\
-\gamma \beta \cdot \nabla \phi^0_\ell \\
\end{pmatrix} (\Lambda_\beta(x)). \label{d.psi0}
\end{align}
Then
\begin{align} &  (\opJ\opH ) Y_k^\pm = \pm \frac{\lambda_k}{\gamma} Y_k^\pm,\label{e.Y}\\ & 
\ker \opH  = \ker (\opJ\opH ) =\Span(\Phi^0_\ell, \ell=1, \dots, \bar \ell).\label{e.ker}\end{align} 
Moreover, 
\begin{equation}\label{o.Y}
\langle \opH Y_k^+,Y_{k'}^+\rangle = \langle \opH Y_k^-,Y_{k'}^-\rangle =0 \quad \hbox{for all $k,k'=1,\ldots,\bar k$.}
\end{equation}
Finally, the family $(Y_k^\pm)_{\pm, k=1,\dots, \bar k}$ is linearly independent. As a consequence, the family $(\opH Y_k^\pm)_{\pm, k=1,\dots, \bar k}$  is linearly independent.

\smallskip

{\rm (iii)} \emph{Exponential decay.} There exist $C_0>0$, $\omega_0 >0$, such that 
for all $\alpha\in \N^d$, $|\alpha|\le 1$, for all $x\in \R^d$,
\begin{equation}\label{expo}
|\partial_\alpha q(x)| + 
|\partial_\alpha \phi_k(x)| + |\partial_\alpha Y_{k}^\pm(x)|+ |\partial_\alpha \phi_\ell^0(x)|
\le C_0 e^{-\omega_0 |x|}.
\end{equation}
\end{lem}

\begin{proof}
 We start by noticing that by rotation (with first vector $\beta/|\beta|$ for $\beta \ne 0$), we can assume that the Lorentz boost is of the form $(\beta,0, \dots, 0)$, where (with slight abuse of notation) $\beta \in (-1,1)$.
 Observe that in this case $\Lambda_\beta(x) = (\gamma x_1, x')$, where $x=(x_1,x'),$ $x' = (x_2, \dots, x_d)$.
%
%Then $\tilde L$ is a compact perturbation of the operator $ - \Delta + (\beta \cdot \nabla)^2 + 1 =-(1-|\beta|^2) \partial_{11}^2 - \Delta'$, where $\Delta'$ is the Laplace operator on the remaining variables $x'= (x_2, \dots, x_d)$. From there, we infer the essential spectrum of $\tilde L$, and the coercivity property \eqref{eq:coer_L}. 

\smallskip

(i) The operator 
$\opL$ is a compact perturbation of $-\Delta + 1$, and so the two operators have the  same essential spectrum $[1,+\infty)$.  In particular, for any $\delta>0$, both operators have a finite number of eigenvalues (counting their multiplicities) on $(-\infty, 1-\delta]$. We define 
$(\phi_k)_{k=1, \dots, \bar k}$,
 $(\lambda_k)_{k=1, \dots, \bar k}$, and $(\phi^0_\ell)_{\ell = 1, \dots, \bar \ell}$
 as in the statement of the lemma.
From the spectral theorem, the following coercivity holds: there exists $c'>0$ such that for any $v \in H^1$ satisfying $(v,\phi_k) = (v,\phi^\ell_0) =0$ for all $k =1, \dots, \bar k$, $\ell = 1, \dots, \bar \ell$, we have
$(\opL v,v) \ge c' \| v \|_{L^2}^2.$
Since $f'(q)$ is bounded, a standard argument proves that 
the coercivity property \eqref{eq:coer_L} holds. 

Note that by direct computations $(\opL q , q) = (1-p) \int |q|^{p+1}<0$ and thus $\bar k\ge 1$. Moreover, it is clear by differentiating \eqref{q} with respect to  $x_j$ that $\partial_j q \in \ker \opL$. Since the family $(\partial_j q)_{j=1,\ldots, d}$ is linearly independent, we obtain $\bar \ell \ge d$.

\smallskip

(ii) Looking for an eigenfunction $Y = (\rho_1,\rho_2)^\top$ of the operator $\opJ\opH $, with eigenvalue $\lambda$,    we are led to   the system
\begin{align*}
\begin{cases}
\beta \partial_1 \rho_1 + \rho_2 = \lambda \rho_1, \\
(\Delta - 1 + f'(q_\beta)) \rho_1 + \beta \partial_1 \rho_2= \lambda \rho_2.
\end{cases}
\end{align*}
The first equation gives  $\rho_2 = (\lambda -\beta \partial_1) \rho_1 $ which we   plug into the second equation:
\[ (- \Delta + 1 - f'(q_\beta) ) \rho_1 +  (\lambda- \beta \partial_1 )^2 \rho_1 =0, \]
which rewrites
\[ -(1-\beta^2) \partial_{11}\rho_1 - \Delta' \rho_1 -2 \lambda \beta \partial_1 \rho_1 +  \rho_1 - f'(q_\beta) \rho_1 = - \lambda^2 \rho_1 \]
where $\Delta'$ is the Laplace operator with respect to the   variable $x' = (x_2, \dots, x_d)$. 
Write \[\rho_1(x) = e^{-\gamma^2 \lambda \beta x_1} \sigma_1 (\gamma x_1, x')
= e^{-\gamma \lambda \beta y_1} \sigma_1 (y), \quad y=(\gamma x_1,x').\] 
Then the equation on $\rho_1$ rewrites  as follows
\begin{multline*}
  e^{-\gamma \lambda \beta y_1} \left[ -(1-\beta^2) \left( \gamma^2  \partial_{11}^2 - 2 \gamma^2 \lambda \beta \gamma \partial_{1} + (\gamma^2 \lambda \beta)^2 \right) \right. \\
 \left. - \Delta'  - 2 \lambda \beta (\gamma \partial_{1} - \gamma^2 \lambda \beta) + (1+\lambda^2) - f'(q) \right] \sigma_1(y) =0, \end{multline*}
  which simplifies as
\[
 -\Delta \sigma_1 + \sigma_1 - f'(q) \sigma_1 =   - \gamma^2 \lambda^2 \sigma_1. 
\]
Therefore $\sigma_1$ has to be an eigenfunction of $\opL$ with   eigenvalue $- \gamma^2 \lambda^2  \le 0$.

Reciprocally, if $\sigma_1 = \phi_k$ and $\lambda = \lambda_k/\gamma$, then  
\[ e^{-\gamma^2 \lambda \beta  x_1} \begin{pmatrix} \phi_k \\
 - \gamma \beta \partial_1 \phi_k + \lambda \gamma^2 \phi_k
\end{pmatrix} (\gamma x_1,x') \]
is an eigenfunction of  $\opJ\opH $ with eigenvalue $\lambda$.

Let us check \eqref{o.Y}:
\begin{align*}
\frac{\lambda_{k'}}{\gamma} \langle \opH Y_k^{+}, Y_{k'}^{+}\rangle &=\langle \opH Y_k^+, \opJ\opH  Y_{k'}^+\rangle =-\langle \opJ\opH  Y_{k}^+, \opH Y_{k'}^+\rangle \\
&= - \frac{\lambda_{k}}{\gamma} \langle Y_k^+,\opH Y_{k'}^+\rangle =- \frac{\lambda_{k}}{\gamma} \langle \opH  Y_k^+,Y_{k'}^+\rangle.
\end{align*}
Since $\lambda_k,\lambda_{k'}>0$, this implies $\langle \opH  Y_k^+,Y_{k'}^+\rangle=0$.
All these computations are similar for $(Y_k^-)_k$.

Also observe that if $\lambda_k \ne \lambda_{k'}$, then $\langle \opH Y_k^{+}, Y_{k'}^- \rangle=0$ with the same argument.

Let us now prove that the $(Y_k^\pm)_{\pm, k=1, \dots, \bar k}$ are linearly independent. Assume that we have the dependence relation:
\[ \sum_{k=1}^{\bar k} ( a_k^+ Y_k^+ + a_k^- Y_k^-) =0. \]
Fix $k_* \in \{ 1, \dots, \bar k \}$, and let $I_{k_*}$ be the set of indices $k \in \{ 1, \dots, \bar k \}$ such that $\lambda_k = \lambda_{k_*}$. As spectral spaces associated to different eigenvalues are in direct sum, we infer that
\[ \sum_{k \in I_{k_*}} a_k^+ Y_k^+ =0, \quad \text{and} \quad \sum_{k \in I_{k_*}} a_k^- Y_k^- =0. \]
In the first equality, the first line writes 
\[ e^{-\gamma \lambda_k \beta \cdot x} \sum_{k \in I_{k_*}} a_k^+ \phi_k (\Lambda_\beta x) =0. \]
As $\Lambda_\beta$ is one-to-one, this means that $\sum_{k \in I_{k_*}} a_k^+ \phi_k =0$ and by linear independence of the $(\phi_k)_{k=1, \dots, \bar k}$, this relation is trivial: $a_k^+=0$ for all $k \in I_{K_*}$, and in particular $a_{k_*}^+ =0$. A similar argument on the second equality gives that $a_{k_*}^- =0$.
Therefore, the dependence relation is trivial, and the $(Y_k^\pm)_{\pm, k=1, \dots, \bar k}$ are linearly independent.

As they are eigenfunctions for $\opJ \opH$ with non zero eigenvalue, we infer that the family $(\opJ \opH Y_k^\pm)_{\pm, k=1, \dots, \bar k}$ is linearly independent. As $\opJ$ is one-to-one (it is an involution), the $(\opH Y_k^\pm)_{\pm, k=1, \dots, \bar k}$ are linearly independent as well.
\smallskip

(iii) The exponential decay of any bound state $q$ and its derivates is well known, and follows from Agmon type estimates; we refer to \cite{GT77}.

By standard elliptic arguments, we first note that 
there exist $C>0$  such that 
for all $\alpha\in \N^d$, $|\alpha|\le 2$, 
\begin{equation}\label{expobis}
\forall x \in \m R^d, \quad |\partial_\alpha \phi_k(x)|\le C  e^{-(1+\lambda_k^2)^{\frac 12} |x|},\quad 
 |\partial_\alpha \phi_\ell^0(x)| \le C e^{-|x|} .
\end{equation}
This, and the definition of $Y_k^{\pm}$ in \eqref{d.Yk} is enough to prove \eqref{expo}.
\end{proof}

\subsection{Spectral analysis of $\opH$}

The eigenfunctions $Y_k^\pm$ of $\opJ\opH$ are related to equation \eqref{eq.V},
as well as the eigenfunctions $\opH Y_k^\pm$  of the adjoint operator $\opH \opJ$.
In particular, it is straightforward to compute the main order time evolution of the projection of the perturbation $V$ on such directions (see Lemma~\ref{le:dota}). However, in order to study stability properties of the flow using energy method (see next Section), the relevant operator 
 turns out to be $\opH$.

The operator $\opH$ is self-adjoint for the $\langle \cdot, \cdot \rangle$ scalar product and we already know from Lemma~\ref{lem:L_JH} that 
\begin{equation} \label{kerH}
\ker \opH = \Span(\Phi^0_\ell, \ell=1, \dots, \bar \ell),\end{equation}
where the vector-valued functions $\Phi^0_\ell$ are defined in \eqref{d.psi0}. However, unlike for $\opJ\opH$, the eigenfunctions of $\opH$ related to negative eigenvalues do not seem to be explicitly related to that of $\opL$.

Nonetheless a key observation of this paper is that for any $\beta\in \R^d$, $|\beta|<1$,  the number of negative directions for the quadratic form $\langle \opH \cdot, \cdot \rangle$ is equal to  the number $\bar k$ of negative eigenvalues of the operator $\opL$.

\begin{lem} \label{lem:H}
The self-adjoint operator $\opH$ has a finite number $\bar m\ge 1$ of negative eigenvalues (counted with multiplicity).
Let $(\Upsilon_m)_{m = 1, \dots, \bar m}$   be an $L^2$-orthogonal family of eigenfunctions of $\opH$ with negative eigenvalues, normalized so that
\begin{align}
 & \opH \Upsilon_m   = -\mu_m^2 \Upsilon_m, \quad \mu_m >0, \quad  m=1, \dots, \bar m \label{eigenH}\\
& \langle \opH \Upsilon_m, \Upsilon_m \rangle  = -1,  \quad \langle \Upsilon_m, \Upsilon_{m'} \rangle = 0 \quad \text{for } m \ne m'. \label{PsiH}
\end{align}
Then the following holds 
\[\bar m=\bar k.\] 
Moreover,   there exists $c>0$ such that for all $V \in H^1 \times L^2$,
\begin{gather} \label{Hcoer} \langle \opH V, V \rangle \ge c\| V \|^2 - \frac{1}{c} \sum_{m=1}^{\bar m} \langle  V, \Upsilon_m \rangle^2 - \frac{1}{c} \sum_{\ell=1}^{\bar \ell} \langle V, \Phi^0_\ell \rangle^2.
\end{gather}
\end{lem}

\begin{proof} 
As before we   assume (without loss of generality) that the Lorentz boost is of the form $(\beta, 0, \dots, 0)$ for some $\beta \in (-1,1)$. Note that
\begin{gather*}
\begin{aligned}
\langle \opH V, V \rangle & = ((- \Delta + 1 - f'(q_\beta) ) v_1, v_1) + 2 \beta  (\partial_1 v_1, v_2 ) + \| v_2 \|_{L^2} ^2 \\ 
& = (\tilde \opL v_1, v_1) + \| \beta \partial_1 v_1 + v_2 \|_{L^2}^2,
\end{aligned} \\
\text{where} \quad \tilde \opL  := -(1-\beta^2) \partial_{11} - \Delta' + 1 - f'(q_\beta)
\end{gather*}
Observe that $\tilde \opL$ is self-adjoint and that it is a compact perturbation of the operator $-(1-\beta^2) \partial_{11}^2 - \Delta'+1$ so it has essential spectrum $[1,+\infty)$. From there we infer that $\opH$ has only finitely many negative eigenvalues, whose eigenfunctions span a vector space of dimension $\bar m$; as $(\Phi^0_\ell)_{\ell=1, \dots, \bar \ell}$ span $\ker \opH$, and this yields \eqref{Hcoer}.

Also notice that if we denote $\tilde V(x) := V(\Lambda_\beta(x))$ then $(\tilde \opL \tilde V)(x) = (\opL V)(\Lambda_\beta(x))$. This means that  a basis of the eigenfunctions of $\tilde \opL$ with negative eigenvalues is given by the  $(\phi_k\circ\Lambda_\beta)_{k =1 \dots, \bar k}$; in particular they span a subspace of dimension $\bar k$.

\smallskip

Now, we prove that $\bar m=\bar k$.
On the one hand,  for $k=1, \dots, \bar k$, define
\[\Phi_k(x)= \begin{pmatrix} 
\phi_k \\
- \gamma \beta \partial_1 \phi_k
\end{pmatrix}(\Lambda_\beta(x))
 \quad \text{so that} \quad \opH \Phi_k = \begin{pmatrix} 
-\lambda_k^2 \phi_k \\
0
\end{pmatrix}(\Lambda_\beta(x)). \]
Then the $(\Phi_k)_k$ are linearly independent as a consequence of the linear independence of the  $(\phi_k)_k$. Let $W \in \Span( \Phi_k, k=1, \dots, \bar k)$ non zero, $W = \sum_k \alpha_k \Phi_k$ and by $L^2$ orthogonality of the $\phi_k$,
\[ \langle \opH W,W \rangle = - \sum_k \lambda_k^2 \alpha_k^2 < 0. \]
Hence $\langle \opH\cdot,\cdot \rangle|_{\Span( \Phi_k, k=1, \dots, \bar k)}$ is definite negative on $\Span( \Phi_k, k=1, \dots, \bar k)$ which is of dimension $\bar k$. By Sylvester inertia theorem, we deduce that $\bar k \le \bar m$.

\smallskip

On the other hand, denote by $\Upsilon_m = ( \upsilon_m^1, \upsilon_m^2 )^\top$ a family of $L^2$-orthogonal eigenfunctions of $\opH$ with negative  eigenvalues $-\mu_m^2$ ($\mu_k>0$),
i.e. $\opH \Upsilon_m = - \mu^2_m \Upsilon_m$. Then $( \upsilon_m^1, \upsilon_m^2 )$ satisfy
\[ \left\{ \begin{array}{r@{}l}
(- \Delta + 1 - f'(q_\beta) ) \upsilon_m^1 - \beta \partial_1 \upsilon_m^2 & = - \mu_m^2 \upsilon_m^1 \\[5pt]
\beta \partial_1 \upsilon_m^1 + \upsilon_m^2  & = - \mu_m^2 \upsilon_m^2
\end{array} \right. \]
so that $\ds \upsilon_m^2 = - \frac{\beta}{1+\mu_m^2} \partial_1 \upsilon_m^1$ and
\[ \left( - \Delta + 1 - f'(q_\beta)  + \frac{\beta^2}{1+\mu_m^2} \partial_{11} \right) \upsilon_m^1 = - \mu_m^2 \upsilon_m^1. \]
Then,
\begin{align*}
\tilde \opL \upsilon_m^1
= - \mu_m^2 \upsilon_m^1 + \frac{\beta^2 \mu_m^2}{1+\mu_m^2} \partial_{11} \upsilon_m^1
\end{align*}
and so
\[
(\tilde \opL\upsilon_m^1,\upsilon_m^1)
=-\mu_m^2\left(  \| \upsilon_m^1 \|_{L^2}^2 + \frac{\beta^2}{1+\mu_m^2} \| \partial_1 \upsilon_m^1 \|_{L^2}^2\right) 
\]
For $m\neq m'$, note first  that the orthogonality $\langle \Upsilon_m,\Upsilon_{m'}\rangle=0$ gives
\begin{align*}
0 = \langle \Upsilon_m,\Upsilon_{m'}\rangle
&=\int \upsilon_m^1 \upsilon_{m'}^1+ \int \upsilon_m^2\upsilon_{m'}^2 \\
&= \int \upsilon_m^1 \upsilon_{m'}^1 + \frac{\beta^2}{(1+\mu_m^2)(1+\mu_{m'}^2)} \int \partial_1 \upsilon_m^1 \partial_1\upsilon_{m'}^1.
\end{align*}
Thus, for $m\neq m'$,
\begin{align*}
(\tilde \opL\upsilon_m^1,\upsilon_{m'}^1) &= 
- \mu_m^2 \left(\int \upsilon_m^1 \upsilon_{m'}^1 +\frac{\beta^2 }{1+\mu_m^2} \int \partial_1 \upsilon_m^1 \partial_1 \upsilon_{m'}^1 \right)\\
&= -\frac{\beta^2 \mu_m^2\mu_{m'}^2}{(1+\mu_m^2)(1+\mu_{m'}^2)} \int \partial_1 \upsilon_m^1 \partial_1 \upsilon_{m'}^1
\end{align*}
Let $w = - \beta^2\sum_{m=1}^{\bar m} \alpha_m \upsilon_m^1$ non zero.  Then, we obtain for $w$
\begin{align*}
(\tilde \opL w, w) & =   - \beta^2 \sum_{m,m'=1}^{\bar m} \frac{\alpha_m\alpha_{m'}\mu_m^2\mu_{m'}^2}{(1+\mu_m^2)(1+\mu_{m'}^2)} \int \partial_1 \upsilon_m^1 \partial_1 \upsilon_{m'}^1 \\ 
& \quad -   \sum_{m=1}^{\bar m} \alpha_m^2\mu_m^2 \left( \| \upsilon_m^1 \|_{L^2}^2 +  \frac{\beta^2}{(1+\mu_m^2)^2}\| \partial_1 \upsilon_m^1 \|_{L^2}^2\right)
\end{align*}
For the first term of the right hand side, we have the identity
\[
  \sum_{m,m'=1}^{\bar m} \frac{\alpha_m\alpha_{m'}\mu_m^2\mu_{m'}^2}{(1+\mu_m^2)(1+\mu_{m'}^2)} \int \partial_1 \upsilon_m^1 \partial_1 \upsilon_{m'}^1= \left\| \sum_{m=1}^{\bar m} \frac{\alpha_m \mu_m^2}{1+\mu_m^2} \partial_1 \upsilon_m^1 \right\|_{L^2}^2,
\]
and we obtain
\begin{align*}
(\tilde \opL w, w) &\le  - \sum_{m=1}^{\bar m} \alpha_m^2 \mu_m^2 \left( \| \upsilon_m^1 \|_{L^2}^2 + \frac{\beta^2}{(1+\mu_m^2)^2} \| \partial_1 \upsilon_m^1 \|_{L^2}^2 \right).
\end{align*}
which means that $(\tilde \opL \cdot, \cdot)$ is  definite negative on $\Span(\upsilon_m^1, m=1, \dots, \bar m)$, a subspace  of dimension $\bar m$. By Sylvester inertia theorem, $\bar m \le \bar k$.
In conclusion, we have proved $\bar m = \bar k$.
\end{proof}

Even if $\bar m = \bar k$, we will still use  $k \in \{ 1, \dots,  \bar k \}$ and $m \in \{ 1, \dots, \bar m \}$ to denote with clarity the ranges of indices of the negative eigenvalues and corresponding eigenfunctions of the operators $\opL$ and $\opH$.

\subsection{A coercivity property}

The main result of this section is the following proposition, which states a coercivity property for $\opH$ related to the eigenfunctions $Y_k^\pm$ and $\Phi^0_\ell$ of $\opJ\opH$.

\begin{prop} \label{prop:coer_H}
There exists $c>0$ such that, for all $V \in H^1 \times L^2$,
\[ \langle \opH V,V \rangle \ge c \| V \|^2 - \frac{1}{c} \sum_{\pm, k=1}^{\bar k} \langle \opH V, Y_k^\pm \rangle^2 - \frac{1}{c} \sum_{\ell=1}^{\bar \ell} \langle V, \Phi^0_\ell \rangle^2. \]
\end{prop}

This result is a generalization of Proposition~2 in \cite{CoMu} to the case of bound states (we also refer to Lemma~5.2 in \cite{DM1} for a previous similar result for the energy critical NLS  equation). In constrast with previous works in the case of ground states, this result is obtained with no a priori information on the spectrum of $\opL$.

\begin{nb}
The proof of the Claim page 7471 of~\cite{CoteMar} is not correct (the identities at the bottom of page 7471 are erroneous), which makes the proof of the proposition invalid.
We provide a corrected proof in the present version.
We refer to \cite{XY} for a self-contained proof in the analogue context of the energy critical wave equation.
We also refer to \cite{ChenJendrej} for an alternate proof.

\end{nb}

\begin{proof}
We  assume   that the Lorentz boost is of the form $(\beta, 0, \dots, 0)$, with $\beta \in (-1,1)$.
We define the functions $W_k$ by
\[
W_k =  Y_k^+ + Y_k^-.
\]
The strategy is to establish that there exists $c>0$ such that for any function $V\in H^1\times L^2$ satisfying the orthogonality conditions
\begin{equation}\label{new:ortho}
\langle \opH V, W_k \rangle = \langle V, \Phi^0_\ell \rangle =0, \quad \mbox{for all $k=1, \dots, \bar k$, $\ell = 1, \dots, \bar \ell$,}
\end{equation}
it holds
\begin{equation}\label{eq:coercH_V}
\langle \opH V,V \rangle \ge c \| V \|^2.
\end{equation}
By standard arguments, as $W_k \in \Span(Y_k^\pm)$, this implies the statement of Proposition~\ref{prop:coer_H}.
Since $\opH$ has exactly $\bar k+\bar \ell$ nonpositive eigenvalues (see Lemma~\ref{lem:H}), it is more satisfactory to obtain coercivity under exactly $\bar k+\bar \ell$ orthogonality conditions as in \eqref{new:ortho} rather than with $2\bar k + \bar \ell$ scalar products as in the statement of the proposition.

\begin{claim}
The functions $(W_k)_{1\leq k\leq \bar k}$ satisfy the following properties.
\begin{enumerate}
\item For all $k\in \{1,\ldots,\bar k\}$, it holds
\begin{equation*}
\langle \opH W_k, W_k \rangle = -\frac{4\lambda_k^2}\gamma.
\end{equation*}
\item For all $k,k'\in \{1,\ldots,\bar k\}$ with $k\neq k'$, it holds
\begin{equation*}
\langle \opH W_k, W_{k'} \rangle = 0.
\end{equation*}
\end{enumerate}
\end{claim}
\begin{proof}[Proof of the Claim]
We recall the expression of $Y_k^{\pm}$
\begin{equation}\label{newdef:Y}
Y_k^\pm(x) = e^{\mp \gamma \lambda_k \beta x_1} \begin{pmatrix}
\phi_k \\
- \gamma \beta\, \partial_1 \phi_k \pm \gamma \lambda_k \phi_k \\
\end{pmatrix} (\gamma x_1, x').
\end{equation}
Proof of (1). From (18), we have
$\langle \opH Y_k^+, Y_k^+ \rangle =\langle \opH Y_k^-, Y_k^- \rangle=0$ and thus
\[
\langle \opH W_k, W_k \rangle
=2 \langle \opH Y_k^+, Y_k^- \rangle.
\]
Using $\opH Y_k^+ = - \frac {\lambda_k}{\gamma} \opJ  Y_k^+$ (see (16))
and then the explicit expressions of $Y_k^\pm$ in~\eqref{newdef:Y}, we compute
\begin{equation*}
\langle \opH Y_k^+, Y_k^- \rangle
=- \frac {\lambda_k}{\gamma} \langle \opJ  Y_k^+,Y_k^-\rangle\\
=- 2\frac{\lambda_k^2 }{\gamma}\langle \phi_k,\phi_k\rangle.
\end{equation*}
The result follows from the normalization $\langle \phi_k,\phi_k\rangle=1$.

Proof of (2). Let $k\neq k'$. In the case where $\lambda_k\neq \lambda_k'$, it is shown in the proof of Lemma~\ref{lem:L_JH} that
$\langle \opH Y_k^+, Y_{k'}^+ \rangle =
\langle \opH Y_k^-, Y_{k'}^-\rangle=
\langle \opH Y_k^+, Y_{k'}^- \rangle=0$, which provides the desired conclusion.

In the case where $\lambda_k= \lambda_k'$,the identities
$\langle \opH Y_k^+, Y_{k'}^+ \rangle =
\langle \opH Y_k^-, Y_{k'}^-\rangle=0$ still hold.
From $\opH Y_k^+ = - \frac {\lambda_k}{\gamma} \opJ  Y_k^+$, then the explicit expressions of $Y_k^\pm$ in~\eqref{newdef:Y}, $\langle\phi_k,\phi_{k'}\rangle=0$ and integration by parts,
we compute
\begin{equation*}
\langle \opH Y_k^+, Y_{k'}^- \rangle
=- \frac {\lambda_k}{\gamma} \langle \opJ  Y_k^+,Y_{k'}^-\rangle\\
= -\frac{2\beta \lambda_k }{\gamma} \langle \partial_1 \phi_k, \phi_{k'}\rangle.
\end{equation*}
Therefore by integration by parts,
\begin{equation*}
 \langle \opH W_k, W_{k'} \rangle
=  \langle \opH Y_k^+, Y_{k'}^- \rangle +  \langle \opH Y_{k'}^+, Y_{k}^- \rangle
= 0,
\end{equation*}
which proves the claim.
\end{proof}

We recall from Lemma~\ref{lem:H} that there exists $c_0>0$ such that, for any $V\in H^1\times L^2$
\begin{equation}\label{eq:lemma2}
\langle \opH V , V \rangle
\geq c_0 \|V\|^2 - \frac 1{c_0} \sum_{k=1}^{\bar k} \langle V,\Upsilon_k\rangle^2
- \frac 1{c_0} \sum_{\ell=1}^{\bar \ell} \langle V,\Phi_\ell^0\rangle^2.
\end{equation}
Recall also from the definition of $\opH$ that, for any $V\in H^1\times L^2$
\begin{multline}\label{eq:HH}
\langle\opH V,V\rangle 
=\|\nabla v_1\|_{L^2}^2+\|v_1\|_{L^2}^2+\|v_2\|_{L^2}^2+2\beta \int (\partial_1 v_1) v_2
-\int f'(q_\beta) v_1^2\\
\geq (1-|\beta|) \|V\|^2 -\int f'(q_\beta) v_1^2.
\end{multline}
Now, we prove that \eqref{new:ortho} implies \eqref{eq:coercH_V}.
For the sake of contradiction, assume that there exists a sequence of functions $V_n=(v_{1,n},v_{2,n})\in H^1\times L^2$
satisfying the orthogonality conditions~\eqref{new:ortho} and the inequality
\[
\langle \opH V_n ,V_n \rangle < \frac 1n  \| V_n \|^2.
\]
By \eqref{eq:HH}, for $n$ large, $\int f'(q_\beta) v_{1,n}>0$ and we may normalize $V_n$ as follows
\[
\int f'(q_\beta) v_{1,n}^2 = 1,
\]
so that the sequence $(V_n)$ is bounded in $H^1\times L^2$. 
Up to extraction of a subsequence, it
converges weakly to a function $V\in H^1\times L^2$
satisfying the orthogonality conditions~\eqref{new:ortho}.
By Rellich Theorem, $\int f'(q_\beta) V^2 = 1$, which means that $V\not\equiv 0$.
Moreover, by
\[
\langle\opH V,V\rangle 
=(1-\beta^2)\|\partial_1 v_1\|_{L^2}^2 + \|\nabla'v_1\|_{L^2}^2+\|v_1\|_{L^2}^2+\|\beta\partial_1v_1+v_2\|_{L^2}^2-\int f'(q_\beta) v_1^2
\]
and weak convergence property, it holds $\langle \opH V ,V \rangle \leq \liminf_{n\to\infty} \langle \opH V_n ,V_n \rangle
\leq 0$.

Now, let
\[
F=\Span \left\{V,W_1,\ldots, W_{\bar k},\Phi_1^0,\ldots,\Phi_{\bar\ell}^0 \right\}.
\]
By the condition \eqref{new:ortho} on $V$ and the properties of the families $(W_k)_{k}$
and $(\Phi_\ell^0)_\ell$, we check that the dimension of $F$ is $1+\bar k+\bar \ell$.
We denote by $\opB$ the restriction of the operator $\opH$ to the subspace $F$.
From what precedes, we check that $\opB$ is nonpositive on $F$.
This is contradictory with \eqref{eq:lemma2} which says that $\opB$ is positive under only
$\bar k+\bar \ell$ orthogonality conditions.
\end{proof}

\section{Proof of Proposition~\ref{pr:3}}

\subsection{Notation}
Let $N\in \N\setminus \{0\}$, $\beta_1,\ldots,\beta_N\in \R^d$ be such that
\[ \forall n, \quad|\beta_n| <1 \quad \text{and} \quad  \forall n' \ne n, \quad \beta_{n'} \ne \beta_n;
\quad \gamma_n := \frac 1{\sqrt{1-\beta_n^2}} ,\]
and let  $q_1,q_2,\ldots,q_N$ be any  bound states of equation \eqref{q}.
Denote by $I$ and $I^0$ the following two sets of indices 
\begin{align*}
I & =\{(n,k): n =1, \dots,  N, \ k = 1, \dots, \bar k_n\},\quad  |I|= \mathop{\mathrm{Card}} I =  \sum_{n=1}^N \bar k_n,\\
I^0 & =\{(n,\ell): n = 1, \dots, N, \ \ell = 1, \dots, \bar \ell_n\}, \quad |I^0|= \mathop{\mathrm{Card}}  I^0 =  \sum_{n=1}^N \bar \ell_n.
\end{align*}
Denote by $\m B$ the closed unit ball of $\R^{|I|}$ for the euclidian norm.
For any $n\in \{ 1, \dots, N \}$, we consider the operators $\opL_n$ and $\opH_n$ for the bound state $q_n$, along with the eigenvalues and eigenfunctions defined in Lemma~\ref{lem:L_JH} $(\lambda_{n,k})_{(n,k) \in I}$, $(\phi_{n,k})_{(n,k)\in I}$ $(\phi^0_{n,\ell})_{(n,\ell)\in I^0}$, $(\Phi_{n,\ell}^0)_{(n,\ell)\in I^0}$ and $(Y_{n,k}^{\pm})_{(n,k)\in I,\pm}$
(with obvious notations).
Let
\begin{gather}
 r_n(t,x) =  q_n(\Lambda_{\beta_n}(x-\beta_nt)),\quad
R_n=
 \begin{pmatrix}  
 r_n
 \\  -\beta_n\cdot \nabla r_n\end{pmatrix}, \\
\label{d:h}
\psi^0_{n,\ell}(t,x)=\phi_{n,\ell}^0(\Lambda_{\beta_n} (x-\beta_nt)),
\quad \Psi^0_{n,\ell}= 
\begin{pmatrix}  \psi^0_{n,\ell} \\  -\beta_n\cdot \nabla \psi^0_{n,\ell}\end{pmatrix} = \Phi_{n,\ell}(x- \beta_n t), \\
%\\&\text{$\psi^0_{n,\ell}$ normalized so that } \|\psi^0_{n,\ell}\|_{L^2}=1
\label{d.Zn}
\text{and} \quad Z_{n,k}^{\pm} (t,x) = (\opH_n Y_{n,k}^{\pm})  (x - \beta_n t)
\end{gather}
be their travelling-in-time counterparts. We recall the equation for $\psi_{n,\ell}^0$:
\begin{equation}\label{Lbetan}
(\Delta - (\beta_n\cdot \nabla)^2 )\psi^0_{n,\ell} - \psi^0_{n,\ell} + f'(r_n) \psi^0_{n,\ell}=0.
\end{equation}

\smallskip

Let $T_0\gg1$ to be fixed later large enough and $\omega>0$ to be fixed later small enough, independently of $T_0$.
For brevity, we will omit to mention the fact that $\omega$ is taken small so that estimates hold.
It will be convenient in the estimates to introduce the following enveloping functions: for any $n=1,\ldots,N$, we set
\[
\rho_n(t,x)=e^{-\omega |x-\beta_n t|},\quad \text{and} \quad \brho = \sum_{n=1}^N \rho_n.
\]
In particular, $\omega$ will be so small and $T_0$ so large that so that for any $n \ne n'$,
\begin{align} \label{def:rho}
\forall t \ge T_0, \forall x \in \m R^d, \quad  e^{-(p_0-1)\omega_0 |x-\beta_n t|} e^{- (p_0-1)\omega_0 |x-\beta_{n'} t|} \le e^{-10 \omega t} \rho(t,x).
\end{align}
(the exponential decay rate $\omega_0>0$ was defined in \eqref{expo}).

\smallskip

Fix any $S_0\ge T_0$.
To prove Proposition~\ref{pr:3}, we show that there exists a choice of coefficients
$(\theta_{n,k}^\pm)_{(n,k)\in I,\pm}$, $|\theta|\ll e^{-\omega S_0}$,
such that
the backward solution $U(t)$ of \eqref{NLKGs} with data
\begin{equation}\label{d.U0}
U(S_0)=\sum_{n=1}^N R_n(S_0) + \sum_{\pm, (n,k) \in I^0} \theta_{n,k}^{\pm}  Z_{n,k}^{\pm}(S_0)
\end{equation}
exists on $[T_0,S_0]$ and satisfies the properties of Proposition \ref{pr:3}.

\smallskip

We consider such a solution $U$ defined on its maximal backwards interval of existence $[S_{\rm max},S_0]$, and we first set
\begin{align}\label{e:dec}
U  = \begin{pmatrix}u \\
\partial_t u\end{pmatrix} = \sum_{n=1}^N R_n  + V,\quad V =  \begin{pmatrix}v \\\partial_t v\end{pmatrix}.
\end{align}
We further decompose $V$ according to the kernel of the linearized operator around each bound state $r_n$. 

\begin{lem}\label{le:Wbis}
For $T_0>1$ large enough and $t \ge T_0$, there exists $b=(b_{n,\ell})_{(n,\ell)\in I^0}$ such that
\begin{equation}\label{d:W}
W = \begin{pmatrix} w \\ z\end{pmatrix}  := V -\sum_{(n,\ell)\in I^0} b_{n,\ell} \Psi^0_{n,\ell}
\end{equation}
satisfies, for all $(n,\ell)\in I^0$, and $C >0$ independent of $t$,
\begin{equation}\label{d:b}
\langle W,\Psi^0_{n,\ell}\rangle = 0,\quad
|b|\le C \|V\|.
\end{equation}
\end{lem}
\begin{proof}
The orthogonality condition in \eqref{d:b} is equivalent to a matrix identity
\[
(\langle V, \Psi^0_{n,\ell}\rangle)_{(n,\ell)\in I^0} = \mathcal H b,
\]
where 
$b=(b_{n,\ell})_{(n,\ell) \in I^0}$ (written in one row) and
\[
\mathcal H = (\langle \Psi^0_{n,\ell},\Psi^0_{n',\ell'}\rangle)_{(n,\ell),(n',\ell')\in I^0}
=\mathcal D^0 + O(e^{-10 \omega t}),
\]
\[
\mathcal D^0= {\rm diag} (\mathcal H_1^0,\ldots,\mathcal H_n^0),\quad
\mathcal H_n^0 = (\langle \Psi^0_{n,\ell},\Psi^0_{n,\ell'} \rangle)_{\ell,\ell'\in (1,\ldots,\bar \ell_n)}.
\]
Note that for fixed $n$, the family $(\Psi^0_{n,\ell})_{\ell\in (1,\ldots,\bar \ell_n)}$ being linearly independent (see Lemma~\ref{lem:L_JH}), 
the Gram matrix $\mathcal H_n^0$ is invertible. Thus, $\q D^0$ is invertible: for $T_0$ large enough, so is $\mathcal H$  and
$b = \mathcal H^{-1} (\langle V, \Psi^0_{n,\ell}\rangle)_{(n,\ell)\in I^0}$ (and $\q H^{-1}$ has uniform norm in $t \ge T_0$).
\end{proof}
Note that \eqref{d:W} is equivalent to
\begin{align}\label{d:wz}
  w & = v - \sum_{n=1}^N \sum_{\ell=1}^{\bar\ell_n} b_{n,\ell} \psi^0_{n,\ell} \\
z &= v_t +  \sum_{n=1}^N \sum_{\ell=1}^{\bar\ell_n} b_{n,\ell} \beta_n \cdot\nabla \psi^0_{n,\ell}
=w_t+  \sum_{n=1}^N \sum_{\ell=1}^{\bar \ell_n} \dot b_{n,\ell}   \psi^0_{n,\ell}
\end{align}
For the sake of brevity, we denote
\[
\br = \sum_{n=1}^N r_n,\quad
\bpsi = \sum_{n=1}^N \sum_{\ell=1}^{\bar\ell_n} b_{n,\ell} \psi^0_{n,\ell}
\quad \text{so that} \quad u=\br + v= \br +\bpsi+w.\] 
Finally, we set
 \begin{equation}\label{d.a}
a^{\pm}_{n,k}=\langle V, Z_{n,k}^\pm\rangle.
\end{equation}

Observe that $\langle \Psi_{n, \ell}^0, Z_{n',k}^\pm \rangle = O(e^{-10 \omega t)})$  for all $n,n'=1, \dots, N$, $\ell =1, \dots, \bar \ell_n$, $k=1, \dots, \bar k_{n'}$ and signum $\pm$ (it is obvious by separation and decay \eqref{expo} when $n \ne n'$, and when $n=n'$ it is equal to $\langle \Phi_{n, \ell}^0, \opH_n \opJ Y_{n',k}^\pm \rangle =\langle \opJ \opH_n \Phi_{n, \ell}^0, Y_{n',k}^\pm \rangle = 0$. Therefore 
\begin{align} \label{eq:W_a}
\langle W, Z_{n,k}^\pm \rangle = a^{\pm}_{n,k} + O(|b| e^{-10 \omega t}).
\end{align}

Let
\[
p_0 = \min(2,p),\quad 1<p_0\le 2.
\]

\subsection{Bootstrap setting}
We consider the following bootstrap estimates
\begin{equation}\label{eq:BS}
\begin{cases}
\|W(t)\|\le e^{- \omega t} ,\quad &
|b(t)|\le e^{- \omega t},\\
 |a^-(t)|\le e^{- \frac 13(p_0+2) \omega t},\quad & |a^+(t)|\le e^{- \frac 13(p_0+2)\omega t}.
\end{cases}
\end{equation}

We claim that  any given initial value of $a^{+}(S_0)$ can be matched by  a suitable choice of initial $\theta$ in the definition of $U(S_0)$ in \eqref{d.U0}.

\begin{lem}\label{le:aa}
There exists a $\mathcal C^1$ map $\Theta:\m B \to (e^{-\frac14(p_0+3) \omega S_0} \m B)^2 $
such that for any $\mathfrak a^+ = (\mathfrak a_{n,k}^{+})_{(n,k) \in I} \in \m B$, if we take
$\theta=(\theta_{n,k}^{\pm})_{\pm,(n,k) \in I}=\Theta(\mathfrak a^+)$ in the definition of $V(S_0)$ from \eqref{d.U0}-\eqref{e:dec}, there holds, for $a^\pm(S_0)=(a^{\pm}_{n,k}(S_0))_{(n,k) \in I}$ (defined in \eqref{d.a}),
\begin{equation}\label{initial}
a^+(S_0) =  e^{- \frac13(p_0+2) \omega S_0} \mathfrak a^+ \quad \text{and} \quad a^-(S_0)=0.
\end{equation}
Moreover,
\begin{equation}\label{initialW}
\|W(S_0)\|\le  e^{-\frac14(p_0+3)  \omega S_0}, \quad |b(S_0)|\le C e^{-10 \omega S_0}.
\end{equation}
\end{lem}
\begin{proof}
The proof of this result is similar to that of Lemma~6 in \cite{CoMu} and Lemma~3 in \cite{CMM}.
In view of \eqref{d.U0} and \eqref{e:dec}, it holds
\[
V(S_0) = \sum_{(n',k')\in I,\pm} \theta_{n',k'}^{\pm}  Z_{n',k'}^{\pm}(S_0)
\]
and so we are looking for a solution $(\theta_{n,k}^\pm)_{\pm, (n,k) \in I}$ of the equalities:
\[
e^{- \frac13(p_0+2) \omega S_0} \mathfrak a^+_{n,k}=a^{+}_{n,k}(S_0)
=\sum_{(n',k')\in I,\pm} \theta_{n',k'}^{\pm}  \langle Z_{n',k'}^{\pm}(S_0),Z_{n,k}^+(S_0) \rangle, \]
\[
0 = a^{-}_{n,k}(S_0)
=\sum_{(n',k')\in I,\pm} \theta_{n',k'}^{\pm}  \langle Z_{n',k'}^{\pm}(S_0),Z_{n,k}^-(S_0) \rangle, \]
which rewrites as
$
a = \mathcal Z \theta,
$ where 
\begin{align*}
a & =(a^+_{1,1},a^-_{1,1},a_{1,2}^+,a_{1,2}^-,\ldots)^\top = (a^\pm)_{\pm, (n,k) \in I}, \\
\theta & =(\theta^+_{1,1},\theta^-_{1,1},\theta_{1,2}^+,\theta_{1,2}^-,\ldots)^\top = (\theta^\pm)_{\pm, (n,k) \in I},
\end{align*}
and where $\mathcal Z$ is the  $2|I| \times 2|I|$  matrix
\[ \mathcal Z=(\langle Z_{n,k}^{\pm}(S_0), Z_{n',k'}^{\pm'}(S_0) \rangle)_{\pm, \pm', (n,k), (n',k')\in I}. \]
In particular, by \eqref{expo}, for $\omega>0$ small enough, we note that
\begin{gather*}
\mathcal Z = \mathcal D  + O(e^{-10 \omega S_0}),\quad \text{where} \quad
\mathcal D ={\rm diag} \begin{pmatrix} \mathcal Z_1,\ldots, \mathcal Z_N \end{pmatrix}, \quad \text{with} \\
\mathcal Z_n  = \begin{pmatrix} 
 (Z_{n,1}^{+}(S_0),Z_{n,1}^+(S_0)) & (Z_{n,1}^{-}(S_0),Z_{n,1}^+(S_0)) &\ldots & (Z_{n,1}^{-}(S_0),Z_{n,\bar k}^+(S_0))\\
  (Z_{n,1}^{+}(S_0),Z_{n,1}^-(S_0)) & (Z_{n,1}^{-}(S_0),Z_{n,1}^-(S_0)) &\ldots & (Z_{n,1}^{-}(S_0),Z_{n,\bar k}^-(S_0))\\
 \vdots & & \ddots \\
 (Z_{n,\bar k}^{+}(S_0),Z_{n,1}^-(S_0)) & (Z_{n,1}^{-}(S_0),Z_{n,1}^-(S_0)) &\ldots & (Z_{n,\bar k_n}^{-}(S_0),Z_{n,\bar k}^-(S_0))
\end{pmatrix}
\end{gather*}
For any $n=1,\ldots,N$, $\mathcal Z_n$ is the Gram matrix of the linearly independent family of size $2\bar k_n$ $(Z_{n,k}^\pm(S_0))_{\pm, k=1,\ldots \bar k_n}$.
Thus, $\mathcal Z_n$ is invertible and $\mathcal D$ is invertible. It follows that $\mathcal Z$ is invertible for $T_0$ large enough.

Moreover, from \eqref{initial}, for $T_0$ large enough,
\begin{align} \label{bd:theta}
|\theta| \le C |a(S_0)| \le C e^{-\frac 13(p_0+2) \omega S_0}.
\end{align}
and so from the definition of $V(S_0)$ above, and the fact that $\langle Z_{n,k}^\pm, \Psi_{n,\ell}^0 \rangle =0$ for any $n$, $k=1, \dots \bar k_n$, $\ell = 1, \dots, \bar \ell_n$ and signum $\pm$, we infer that $|\langle V(S_0), \Psi_{n,\ell}^0 \rangle| \le C e^{-10 \omega S_0}$.

Recalling the definition of $b$ (at the end of the proof of Lemma \ref{le:Wbis}), we deduce that 
\[ |b(S_0)| = |\q H^{-1} (\langle V, \Psi_{n,k}^0)_{(n,\ell) \in I^0})| \le C | (\langle V, \Psi_{n,k}^0)_{(n,\ell) \in I^0} | \le C e^{-10 \omega S_0}. \]
From \eqref{bd:theta} and \eqref{d:b}, we get
\[ \| V(S_0)\|, \| W(S_0) \| \le C e^{-\frac 13(p_0+2) \omega S_0}. \]
To conclude, simply observe that for large $T_0$, $C e^{-\frac 13(p_0+2) \omega S_0} \le e^{-\frac 14(p_0+3) \omega S_0}$.
\end{proof}

We define the following backward exit time $S_\star = S_\star(\mathfrak a^+)$ related to the bootstrap estimates \eqref{eq:BS}.
\[
S_\star = \inf\{ T\in [T_0,S_0] \text{ such that $U$ is defined  and satisfies \eqref{eq:BS} on $[T,S_0]$}\}.
\]
Note that in view of Lemma \ref{le:aa}, $U(S_0)$ satisfies \eqref{eq:BS} so that that $T_0\le S_\star\le S_0$ is well-defined.
Our goal is to find a specific choice of $\mathfrak a^+ \in \m B$ so that $S_\star=T_0$.

The argument goes by contradiction of this condition. 

In the next subsections, we fix a choice of $\mathfrak a^+ \in \m B$, such that
\begin{equation}\label{contra}
T_0< S_\star(\mathfrak a^+) \le S_0.
\end{equation}
We now derive estimates on $\|W\|$, $|b|$ and $|a^\pm|$ on $[S^\star,S_0]$, so as to prove -- in Lemma \ref{le:tra} -- that the flow issued from $\mathfrak a^+$ is transverse at the exit time $S_\star = S_\star(\mathfrak a^+)$.

\subsection{Equation of $W$ and preliminary estimates}
\begin{lem}\label{le:W}
The function $W$ satisfies  
\begin{equation}
\partial_t W
  =  \begin{pmatrix} z   \\  \Delta w -w+ f\left(\br + \bpsi+w\right)
-f\left(\br + \bpsi\right) \end{pmatrix}    + 
\sum_{(n,\ell) \in I^0} \dot b_{n,\ell} \Psi^0_{n,\ell} +  G\label{NLKGw} \end{equation}
where
\begin{equation}\label{d:g}
G = \begin{pmatrix}0 \\ \bg \end{pmatrix}, \quad 
\bg :=f\left(\br + \bpsi\right) - \sum_{n=1}^N f(r_n)
- \sum_{(n,\ell) \in I^0} b_{n,\ell} f'(r_n) \psi^0_{n,\ell}.
\end{equation}
\end{lem}
\begin{proof}
First, since $U$ and $R_n$ solve \eqref{NLKGs}, it is direct to check the following equation for $v$
\begin{align}\label{e:V}
\partial_t^2  v = \Delta v-v + f\left(\br + v \right)- \sum_{n=1}^N f(r_n)  .
\end{align}
Next, the first line of \eqref{NLKGw} follows from  the definition of $z$.
For the second line, we observe from the equation of $V$,
\begin{align*}
\partial_t z 
& = \partial_t^2 v + \sum_{(n,\ell) \in I^0} \dot b_{n,\ell} \beta_n \cdot \nabla \psi^0_{n,\ell} - \sum_{(n,\ell) \in I^0} b_{n,\ell} (\beta_n\cdot \nabla)^2 \psi^0_{n,\ell} \\
& = \Delta v - v + f(\br + v) - \sum_{n=1}^N f({r_n}) + \sum_{(n,\ell) \in I^0} \dot b_{n,\ell} \beta_n \cdot \nabla \psi^0_{n,\ell} \\
 & \qquad - \sum_{(n,\ell) \in I^0} b_{n,\ell} (\beta_n\cdot \nabla)^2 \psi^0_{n,\ell}.
\end{align*}
Inserting $v = w + \sum_{(n,\ell) \in I^0} b_{n,\ell} \psi^0_{n,\ell}=w+ \bpsi$, 
and using \eqref{Lbetan}, we find the second line of \eqref{NLKGw}.
\end{proof}

Now, we derive some preliminary estimates related to the equation of $W$. Recall that $p_0 = \min(2,p)$, $1<p_0\le 2$.

First, note that for $K >0$ and any real numbers $(s_j)_{j=1,\ldots,\bar j}$ such that $|s_j|\le K$, the following holds
\begin{align}
&\Big|f\Big(\sum_{j=1}^{\bar j} s_j \Big)-\sum_{j=1}^{\bar j} f(s_j)\Big| \le C(K) \sum_{j\neq j'} |s_j|^{p_0-1}|s_{j'}|^{p_0-1},\label{e:un}\\
&|f(s_1+s_2)-f(s_1) - s_2 f'(s_1)|\le C(K) |s_2|^{p_0},\label{e:undeux}\\
&|f'(s_1+s_2)- f'(s_1)|\le C(K) |s_2|^{p_0-1}.\label{e:untrois}
\end{align}
Second, applying these estimates to various situations, using \eqref{expo} and \eqref{def:rho}, we obtain
\begin{equation}\label{e:deux}
\Big|f( \br )- \sum_n f(r_n)\Big|
\le C \sum_{n\neq n'}|r_n|^{p_0-1} |r_{n'}|^{p_0-1} \le C e^{-10 \omega t} \brho,
\end{equation}
\begin{equation}\label{e:trois}
\Big|f(\br + \bpsi)-f(\br )- f'(\br ) \bpsi \Big|
\le  C |\bpsi|^{p_0} \le C |b|^{p_0} \brho,
\end{equation}
\begin{equation}\label{e:quatre}
\Big| f'(\br) \bpsi - \sum_{(n,\ell) \in I^0} b_{n,\ell} f'(r_n) \psi^0_{n,\ell}\Big|\le 
 \sum_{(n,\ell) \in I^0} \Big| \psi^0_{n,\ell} b_{n,\ell} (f'(\br )-f'(r_n)) \Big| 
 \le C |b|e^{-10 \omega t} \brho.
\end{equation}
(the implicit constant does essentially depend on $\max( \| q_n \|_{L^\infty}, n=1, \dots, N)$). In particular, we obtain
\begin{equation}\label{sur.g}
|\bg (t) |\le C \left( e^{-10 \omega t} + |b(t)|^{p_0} \right) \brho (t) .
\end{equation}
Moreover, we also have
\begin{equation}\label{sur.w1}
|f(\br +\bpsi+w) - f(\br + \bpsi)|\le C \left( \brho^{p_0-1}|w|^{p_0-1}  +|w|^p\right),
\end{equation}
\begin{equation}\label{sur.w2}
|f(\br +\bpsi+w) - f(\br +\bpsi)-w f'(\br +\bpsi)|
\le C \left(|w|^{p_0} + |w|^p\right).
\end{equation}
Since $p>1$, a similar estimate for $F$ holds:
\begin{equation}\label{sur.F}
\left| F\left(\br +\bpsi +w\right)-F\left(\br + \bpsi \right)
-w f\left(\br +\bpsi\right) \right| \le C \left(|w|^2 + |w|^{p+1}\right).
\end{equation}
\subsection{Degenerate directions. Estimates for  $(b_{n,\ell})$} 
\begin{lem}\label{le:dotb}
For all $t\in [S_\star,S_0]$,
\begin{equation}\label{dotb}
 |\dot b(t) |\le C \left( \|W (t) \| + e^{-10\omega t} +  |b(t)|^{p_0}\right)
 \le C e^{-\omega t}.
\end{equation}
\end{lem}
\begin{proof}
We differentiate the orthogonality $\langle W,\Psi^0_{n,\ell}\rangle=0$ from \eqref{d:b},
using \eqref{NLKGw},
\begin{align*}
0 
= \frac {d}{dt} \langle W,\Psi^0_{n,\ell}\rangle  &
= \langle \partial_t W,\Psi^0_{n,k}\rangle - \beta_n \langle W, \nabla \Psi^0_{n,k}\rangle\\
& = 
\left\langle\begin{pmatrix} \beta_n \cdot\nabla w + z
\\  \Delta w  -w + f(\br +\bpsi+w)-f(\br +\bpsi) + \beta_n \cdot\nabla z\end{pmatrix} ,\Psi^0_{n,k}\right\rangle \\
& + \sum_{(n',\ell')\in I^0} \dot b_{n',\ell'} \langle \Psi^0_{n',\ell'},\Psi^0_{n,\ell}\rangle + \langle G, \Psi^0_{n,k}\rangle.
\end{align*}
We see that the first term of the equality is bounded by $C \| W \|$, by performing integration by parts so that derivatives fall on the components of $\Psi^0_{n,k}$ and using \eqref{e:un}. Using \eqref{sur.g} and  the notation of the proof of Lemma~\ref{le:Wbis}, we obtain
\[
|\mathcal H b | \le C \left( \|W\| + e^{-10\omega t} +  |b|^{p_0}\right),
\]
and the first estimate in \eqref{dotb} follows from the the matrix $\mathcal H^{-1}$ being uniformly bounded, and the second, from the bootstrap estimate \eqref{eq:BS}.
\end{proof}

\subsection{Energy properties}
We let
\[
 \mathcal E_W 
  = \frac 12 \int   |z|^2 + |\nabla w|^2+|w|^2
  -2 \left[F\left(\br + \bpsi+w\right) -  F\left(\br + \bpsi \right) - w f\left(\br + \bpsi\right)\right]
\]
We consider a $C^\infty$ radial function $\chi:\R^d\to\R$ such that
\begin{align}\label{d:chi}
\chi(x)=0 \text{ for $|x|\ge 2$},\quad
\chi(x)=1 \text{ for $|x|< 1$},\quad
0\le \chi \le 1 \text{ on $\R^d$}.
\end{align} 
We set
\begin{align}\label{d:Pk}
\mathcal P_n= \frac 12 \int \chi_n  z \nabla w ,\quad
\chi_n(t,x)=\chi\left(\frac {x-\beta_n t}{\delta t}\right),
\end{align}
where
\[
\delta = \frac 1{10} \min \{|\beta_n - \beta_{n'}| : \ 1 \le n,n' \le N , \ n \ne n' \}
\]
and
\begin{align}
\mathcal F(t) = \mathcal E_W(t) + 2 \sum_{n=1}^N \beta_n\cdot \mathcal P_n(t).
\end{align}
\begin{lem}\label{le:E}  For all $t\in [S_\star,S_0]$,
\begin{equation}\label{e:FV}
\left| \frac{d}{dt} \mathcal F (t) \right| \le \frac {C}t e^{-2 \omega t}, \end{equation}
\begin{equation}\label{c:F}
 \|W(t)\|^2 \le C\left ( \mathcal F(t) +|a(t)|^2 +e^{-10\omega t}\right).
\end{equation}
\end{lem}
\begin{proof}\emph{Proof of \eqref{e:FV}.}
First, we see that
\begin{align*}
\frac {d\mathcal E_W}{dt} & =
\int z \partial_t z +\int \partial_t w \left[- \Delta w +w - f\left(\br +\bpsi +w \right) + f\left(\br +\bpsi\right)\right]\\
&\quad -  \int \partial_t (\br +\bpsi)\left[f\left(\br +\bpsi +w \right) - f\left(\br +\bpsi \right)- w f'\left(\br +\bpsi\right)\right]\end{align*}
Using \eqref{NLKGw},
\begin{align*}
\int z \partial_t z & = \int (\partial_t w + \sum_{n,\ell} \dot b_{n,\ell} \psi^0_{n,\ell}) [\Delta w - w + f(\br +\bpsi+w) - f(\br +\bpsi)]\\
& \quad + \int z (\sum_{n,\ell} \dot b_{n,\ell} \beta_{n} \cdot \nabla \psi^0_{n,\ell} + \bg ).\end{align*}
Now, we claim
\begin{align}\label{e:Ln}
\MoveEqLeft \int \psi^0_{n,\ell} [\Delta w - w + f(\br +\bpsi+w) - f(\br +\bpsi)] \nonumber\\
& = \int w(\beta_n\cdot \nabla)^2 \psi^0_{n,\ell} +O( \|w\|_{L^2}^{p_0}+\|w\|_{L^2} \left( e^{-10\omega t} + |b|) \right)).
\end{align}
Indeed, using \eqref{Lbetan},
\begin{align*}
\MoveEqLeft \int \psi^0_{n,\ell} [\Delta w - w + f(\br +\bpsi+w) - f(\br +\bpsi)]\\
& = \int \psi^0_{n,\ell} [\Delta w - w + w f'(r_{n})] + \int \psi^0_{n,\ell} \left[f(\br +\bpsi+w) - f(\br +\bpsi) - w f'(r_{n})\right]\\
& = \int w(\beta_n\cdot \nabla)^2 \psi^0_{n,\ell} + \int \psi^0_{n,\ell} \left[f(\br +\bpsi+w) - f(\br +\bpsi) - w f'(r_{n})\right].
\end{align*}
Moreover, by \eqref{sur.w2} and \eqref{e:un},
\begin{align*}
\MoveEqLeft \left| \int \psi^0_{n,\ell} \left[f(\br +\bpsi+w) - f(\br +\bpsi) - w f'(r_{n})\right]\right| \\
& \le C  \int \rho_n \left| f(\br +\bpsi+w) - f(\br +\bpsi) - w f'(\br +\bpsi) \right| \\
& \quad  + C\int \rho_n  |w| \left|  f'(\br +\bpsi)-f'(\br )\right| + C \int \rho_n  |w| \left|  f'(\br )-f'(r_{n})\right| \\
&  \le C\|w\|_{L^2}^{p_0}+C \|w\|_{L^2} \left( e^{-10\omega t} + |b|) \right),
\end{align*}
which proves \eqref{e:Ln}. 

Next, using \eqref{sur.g}, we have
\begin{equation}\label{zgRzGN}
\left|\int z \bg \right| \le C (e^{-10\omega t} + |b|^{p_0} )\|z\|_{L^2}.
\end{equation}
Thus,
\begin{align*}
\MoveEqLeft \int z \partial_t z +\int \partial_t w \left[- \Delta w +w - f\left(\br +\bpsi +w \right) + f\left(\br +\bpsi\right)\right]\\
& =  \int  w\left( \sum_{(n,\ell) \in I^0} \dot b_{n,\ell}     (\beta_n\cdot \nabla)^2 \psi^0_{n,\ell} \right)
+\int z \left(\sum_{(n,\ell) \in I^0} \dot b_{n,\ell} ( \beta_{n} \cdot \nabla \psi^0_{n',\ell})\right)
\\ 
& \quad + O\left(|\dot b| \left[ \|w\|_{L^2}^{p_0}+\|w\|_{L^2} \left( e^{-10 \omega t} + |b|\right)\right]\right) 
 +O\left( \|z\|_{L^2}(e^{-10\omega t} + |b|^{p_0} ) \right).
\end{align*}
Finally, we have
\begin{align*}
\MoveEqLeft -\int \partial_t \br \left[f\left(\br +\bpsi +w \right) - f\left(\br +\bpsi \right)- w f'\left(\br +\bpsi\right)\right]\\
& =  
\sum_{n=1}^N \int (\beta_n \cdot \nabla r_n) \left[f\left(\br +\bpsi +w  \right) - f\left(\br +\bpsi  \right)- w f'\left(\br +\bpsi \right)\right],\end{align*}
and, since $\partial_t(b_{n,\ell} \psi^0_{n,\ell})  = \dot b_{n,\ell} \psi^0_{n,\ell}
- b_{n,\ell} (\beta_n \cdot \nabla  \psi^0_{n,\ell})$, by \eqref{sur.w2},
\begin{align*}
 -\int \partial_t \bpsi \left[f\left(\br +\bpsi +w \right) - f\left(\br +\bpsi \right)- w f'\left(\br +\bpsi\right)\right]  = O\left((|b|+|\dot b|) \|W\|^2 \right).
 \end{align*}

Thus, in conclusion, using also \eqref{dotb} to control $|\dot b|\le C\left(\|W\|+ e^{-10\omega t}+|b|^{p_0}\right)$,
\begin{align}\label{e:dtEV}\begin{aligned}
\frac {d\mathcal E_W}{dt} & = 
\int  w\left( \sum_{(n,\ell) \in I^0} \dot b_{n,\ell}     (\beta_n\cdot \nabla)^2 \psi^0_{n,\ell} \right)
+\int z \left(\sum_{(n,\ell) \in I^0} \dot b_{n,\ell} ( \beta_{n} \cdot \nabla \psi^0_{n,\ell})\right)
\\
& \quad +\sum_n \int (\beta_n \cdot \nabla r_n) \left[f\left(\br +\bpsi +w  \right) - f\left(\br +\bpsi \right)- w f'\left(\br +\bpsi \right)\right] \\ 
 & \quad + O\left(   \|W\|^{p_0+1}+  |b|  \|W\|^{p_0}  + (e^{-10\omega t} + |b|^{p_0} ) \|W\| \right).
\end{aligned}\end{align}

Now, we compute
\begin{align*}
\frac{d \mathcal P_n}{d t}   & =  
\frac 12 \int (\partial_t \chi_n ) z \nabla w + \frac 12 \int \chi_n  z   \nabla \partial_t w  
 + \frac 12 \int \chi_n   \nabla w \, \partial_t z\\
    &  = -\frac 12 \int\left( \frac x{t}\cdot \nabla \chi_n\right) z \nabla w
    + \frac 12\int \chi_n  z \nabla z - \frac 12 \int \chi_n z \left(\sum_{n',\ell} \dot b_{n',\ell}\nabla \psi^0_{n',\ell} \right)\\
    & \quad  +\frac 12 \int \chi_n\Delta w  \nabla w 
    -\frac 12 \int \chi_n w \nabla w
  +\frac 12 \int \chi_n \nabla w \left( f\left(\br +\bpsi +w \right)-f\left(\br +\bpsi \right)\right)  \\&  \quad
  +\frac 12 \int \chi_n \nabla w \left(\sum_{n',\ell} \dot b_{n',\ell} (\beta_{n'}\cdot \nabla \psi^0_{n',\ell})\right) +\frac 12 \int \chi_n \bg   \nabla w.
   \end{align*}
Integrating by parts, this writes
\begin{align*}
\frac{d \mathcal P_n}{d t} 
&  = -\frac 12 \int\left( \frac x{t}\cdot \nabla \chi_n\right) z \nabla w
-\frac 1 4 \int  \nabla \chi_n  z^2-\frac 14 \int \nabla \chi_n |\nabla w|^2 \\ 
& \quad + \frac 12 \int( \nabla \chi_n  \cdot\nabla w)\nabla w
+ \frac 14 \int w^2 \nabla \chi_n \\
& \quad-\frac 12 \int \nabla \chi_n \left[ F\left(\br +\bpsi +w\right)-F\left(\br +\bpsi \right)-w f\left(\br +\bpsi \right)\right] \\
& \quad-\frac 12\int  \chi_n \nabla (\br +\bpsi)\left[f\left(\br +\bpsi +w\right) - f\left(\br +\bpsi \right)- w f'\left(\br +\bpsi \right)\right]\\ 
& \quad - \frac 12 \int \chi_n z \left(\sum_{n',\ell} \dot b_{n',\ell}\nabla \psi^0_{n',\ell} \right)
+\frac 12 \int \chi_n \nabla w \left(\sum_{n',\ell} \dot b_{n',\ell} (\beta_{n'}\cdot \nabla \psi^0_{n',\ell})\right)\\
& \quad +  \frac 12 \int \chi_n  \bg  \nabla w.
 \end{align*}
For the terms on the first 3 lines, we use \eqref{sur.F} to bound
\begin{align*}
\MoveEqLeft \left| \int\left( \frac x{t}\cdot \nabla \chi_n\right) z\nabla w \right|+
\left| \int  \nabla \chi_n  z^2\right|+
\left| \int \nabla \chi_n |\nabla w|^2\right|+
\left|\int( \nabla \chi_n  \cdot\nabla w)\nabla w \right|\\
\MoveEqLeft +
\left| \int \nabla \chi_n  w^2\right|+
\left| \int \nabla \chi_n \left[ F\left(\br +\bpsi +w\right)-F\left(\br +\bpsi \right)-w f\left(\br +\bpsi\right)\right]\right| \\
& \le C \|\nabla \chi_n\|_{L^\infty}  \int \left(z^2 + |\nabla w|^2 + w^2 + |w|^{p+1} \right) \le \frac{C}{t} \|W\|^2,
\end{align*}
For the fourth line, using \eqref{sur.w2}, we have
\begin{align*}
\MoveEqLeft \left|\int  (1-\chi_n) \nabla r_n \left[f\left(\br +\bpsi +w \right) - f\left(\br +\bpsi \right)- w f'\left(\br +\bpsi \right)\right]\right|  \\
    & \le C  \|\nabla q_{n,\beta_n}\|_{L^\infty( |x| > \delta t )}\|W\|^{p_0}
    \le C e^{-10 \omega t}   \|W\|^{p_0},
\end{align*}
and for $n'\neq n$, using \eqref{expo},
\begin{align*}
\left| \int \chi_n z    \nabla \psi^0_{n',\ell}  \right| + \left| \int \chi_n \nabla w      (\beta_{n'}\cdot \nabla \psi^0_{n',\ell})\right|\le C e^{-10\omega t}   \|W\|.
\end{align*}
Moreover, by \eqref{sur.g}
\[
\left|\int \chi_n \bg  \nabla w  \right| \le C ( e^{-10\omega t} + |b|^{p_0}) \| W \|
\]
Thus, in conclusion for this term
\begin{align}\label{e:Pk}\begin{aligned}
  \frac{d \mathcal P_n}{d t}   &=- \frac 12 \int  z \left(\sum_{\ell=1}^{\bar \ell_n} \dot b_{n,\ell}\nabla \psi^0_{n,\ell} \right)
    +\frac 12 \int  \nabla w \left(\sum_{\ell=1}^{\bar \ell_n} \dot b_{n,\ell} (\beta_{n}\cdot \nabla \psi^0_{n,\ell})\right)\\ &\quad
+\frac 12\int  \nabla r_n \left[f\left(\br +\bpsi +w \right) - f\left(\br +\bpsi  \right)- w f'\left(\br +\bpsi \right)\right]  \\
&\quad +O\left( \frac 1t \|W\|^2+ ( e^{-10\omega t} + |b|^{p_0} ) \|W\|\right)
\end{aligned}\end{align}
Combining \eqref{e:dtEV} and \eqref{e:Pk}, we find
\begin{align}\label{e:FVbis}
\left| \frac{d \mathcal F}{d t} \right| \le C\left(
   \|W\|^{p_0+1}+  \left( |b| +t^{-1}\right) \|W\|^2  
 + (e^{-10\omega t} + |b|^{p_0} ) \|W\| \right) \end{align}
Using \eqref{eq:BS}, we find \eqref{e:FV}.

\smallskip

\emph{Proof of \eqref{c:F}.}

Expanding $\q F(t)$ we get that
\[ \q F(t) = \sum_{n=1}^N \langle \tilde \opH_n(t) W,W \rangle  + O(\| W \|^{p_0+1} + e^{-10 \omega t}), \]
 where $\tilde \opH_n(t)$ is the analog of $\opH_n$, localized on the ball $B(0,\delta t)$ and translated by $\beta_n t$. Using standard localization arguments and Proposition~\ref{prop:coer_H}, we infer the following property 
\begin{equation}\label{c:Fbis}
 \|W(t)\|^2 \le C\left ( \mathcal F(t) +\sum_{(n,k)\in I,\pm} |\langle W,Z_{n,k}^\pm\rangle |^2
 +\sum_{(n,\ell)\in I^0} |\langle W,\Psi^0_{n,\ell}\rangle |^2  + e^{-10 \omega t} \right) .
\end{equation}
We refer e.g. to \cite{MMT} for further details.

Recall that by construction \eqref{d:b}, $\langle W,\Psi^0_{n,\ell}\rangle=0$. Finally by \eqref{eq:W_a} (and the boostrap \eqref{eq:BS}), we have
\begin{gather*}
|\langle W,Z_{n,k}^\pm\rangle | \le |a_{n,k}^\pm(t)| +  C |b(t)| e^{-10\omega t} \le C (|a|+ e^{-10\omega t}). \qedhere
\end{gather*}
\end{proof}

\subsection{Negative directions. Transversality at $S_\star$}
Recall that we have set $
a^{\pm}_{n,k}=\langle V,Z^{\pm}_{n,k}\rangle.$
\begin{lem}[Negative directions]\label{le:dota}
For all $t\in [S_\star,S_0]$,
\begin{equation}\label{dota}
\left| \dot a^{\pm}_{n,k} (t) \pm \frac{\lambda_{n,k}}{\gamma_n} a^{\pm}_{n,k}(t)\right|
\le  C e^{- \frac{1}{2} (p_0+1) \omega t}.
\end{equation}
\end{lem}
\begin{proof}
From \eqref{e:V}, we rewrite the equation of $V$ as follows
\[
\partial_t V
  =  \begin{pmatrix} \partial_t v   \\  \Delta v -v+ f\left(\br + v\right)
-f\left(\br \right) \end{pmatrix}  
=  \begin{pmatrix} 0 & 1  \\  \Delta  -1+ f'\left(r_n\right) & 0 \end{pmatrix} V+G_n,
\]
where
\[
G_n = \begin{pmatrix} 0  \\ g_n  \end{pmatrix},\quad g_n = f\left(\br + v\right)-f\left(\br \right)-f'(r_n) v.
\] 
Then, by \eqref{e.Y}
\begin{align*}
\frac {d}{dt} a_{n,k}^\pm &
= \langle \partial_t V,Z_{n,k}^\pm\rangle - \beta_n \langle V, \nabla Z_{n,k}^\pm\rangle\\
& = \left\langle  \begin{pmatrix} \beta_n \cdot\nabla & 1  \\  \Delta  -1+ f'\left(r_n\right) & \beta_n \cdot\nabla\end{pmatrix} V
,Z_{n,k}^\pm\right\rangle  + \langle G_n, Z_{n,k}^\pm\rangle\\
& = \left\langle \opJ \opH_n V(.+\beta_n t)
, \opH_n Y_{n,k}^\pm\right\rangle  + \langle G_n, Z_{n,k}^\pm\rangle\\
& = - \left\langle  V(.+\beta_n t)
, \opH_n (\opJ \opH_n Y_{n,k}^\pm) \right\rangle  + \langle G_n, Z_{n,k}^\pm\rangle
= \mp \frac{\lambda_{n,k}}{\gamma_n} a_{n,k}^{\pm}  + \langle G_n, Z_{n,k}^\pm\rangle.
\end{align*} 
Now, we estimate $\langle G_n, Z_{n,k}^\pm\rangle$. By \eqref{e:un} and \eqref{e:undeux}, one has
\begin{align*}
|g_n| & = |f(\br +v) - f(r_n) - f'(r_n) v|\\
&\le |f(\br +v)-f(r_n+v)| + |f(r_n+v)-f(r_n)-f'(r_n)v|\\
&\le \sum_{n', n'\ne n} |f(r_n')| + C e^{-10 \omega t} \brho + \sum_{n', n'\ne n} |v|^{p_0-1} |r_{n'}|^{p_0-1}
+ C |v|^{p_0}.
\end{align*}
Thus,
\[
|\langle G_n, Z_{n,k}^\pm\rangle|\le C \|V\|^{p_0} + C e^{-10 \omega t}
\le C \|W\|^{p_0}+ C |b|^{p_0} + C e^{-10 \omega t}.
\]
Therefore, we have obtained, using the bootstrap estimates \eqref{eq:BS} for the final bound,
\begin{align*}
\left| \dot a^{\pm}_{n,k}   \pm \frac{\lambda_{n,k}}{\gamma_n} a^{\pm}_{n,k} \right|
  &\le C \left(  \|W\|^{p_0} + |b|^{p_0} + C e^{-10\omega t}\right)\le  C e^{- \frac 1 2 (p_0+1) \omega t}. \qedhere
\end{align*}
\end{proof}

We close the estimates for $\|W\|$, $|b|$ and $|a^-|$ in the following result.
\begin{lem}
For all $t\in [S_\star,S_0]$,
\begin{equation}\label{close}
\|W(t)\|\le \frac 12 e^{- \omega t},\quad |b(t)|\le \frac 12 e^{- \omega t},\quad
|a^-(t)|\le \frac 12 e^{-\frac{1}{3} (p_0+2) \omega t}.
\end{equation}
Moreover,
\begin{equation}\label{saturation}
|a^+(S_\star)| = e^{-\frac{1}{3} (p_0+2) \omega S_\star}.
\end{equation}
\end{lem}
\begin{proof}
Let $t\in [S_\star,S_0]$.
First,  
%we integrate \eqref{dotb} on $[t,S_0]$, and using \eqref{initialW}, we find
%\[
%|b(t)|\le C (e^{-\omega t} + e^{-10\omega S_0})  \le \frac 12 e^{- \frac 12 (1+\frac 1{p_0})\omega t}.
%\]
%Then 
$\q F(0) \le C \| W(S_0) \|^2$ so that with \eqref{initialW}, integrating \eqref{e:FV} on $[t,S_0]$ yields
\[ \q F(t) \le \frac{C}{t} e^{-\omega t}. \]
It follows from \eqref{c:F} and the bootstrap assumption \eqref{eq:BS} that
\begin{align*}
\|W(t)\|^2 & \le C \mathcal F(t) + C |a(t)|^2  + C e^{-10 \omega t} \\
& \le \frac C{t} e^{-2 \omega t}+ C e^{- \frac 2 3 (p_0+2) \omega S_0} + C e^{-10 \omega t}  \le  \frac C{t} e^{-2 \omega t}.
\end{align*}
For $T_0 \ge 4C$ large enough, we get $\ds \| W(t) \| \le \frac{1}{2} e^{-\omega t}$.

Then we integrate \eqref{dotb} on $[t,S_0]$, using \eqref{initialW} and our improved bound $\| W(t) \| \le \frac{C}{\sqrt{t}} e^{-\omega t}$, and we find
\[ |b(t)|\le C \left( \frac{1}{\sqrt{t}} e^{-\omega t} + e^{-10\omega S_0} \right)  \le \frac 12 e^{- \omega t}. \]

In view of \eqref{initial} and \eqref{dota}, we have
\begin{align*}
\left| \frac{d}{dt} \left[ e^{-\frac{\lambda_{n,k}}{\gamma_n} t}    a^{-}_{n,k} (t) \right] \right|
\le  C e^{-\frac{\lambda_{n,k}}{\gamma_n} t -  \frac 1 2 (p_0+1) \omega t},\quad
 a^{-}_{n,k} (S_0)=0.
\end{align*}
Thus, by integration on $[t,S_0]$, $|a^{-}_{n,k} (t)| \le C e^{- \frac 1 2 (p_0+1) \omega t}$. Taking the $\ell^2(\m R^{|I|})$ norm, we get
\[ |a^-(t)| \le \frac 12 e^{-\frac{1}{3} (p_0+1)  \omega t}.
\]
From the contradiction assumption \eqref{contra} and a continuity argument, it follows that there must be at least one equality in the bootstrap assumptions \eqref{eq:BS} at time $S_\star$: in view of the above, the only possibility is $|a^+(S_\star)| = e^{- \frac{1}{3} (p_0+2) \omega S_\star}$.
\end{proof}

Now, we are in a position to state the transversality condition on $a^+$ at $S_\star$.
\begin{lem}\label{le:tra}
Let $\mathcal A(\mathfrak a^+,t)=e^{\frac 23 (p_0+2) \omega t}|a^+(t)|^2$.
Then
\begin{equation}
\frac{\partial}{\partial t} \mathcal A(\mathfrak a^+,t) |_{t=S_\star(\mathfrak a^+)} <0.
\end{equation}
\end{lem}
\begin{proof} 
Let $\ds c_0 =  \min \left( \frac{\lambda_{n,k}}{\gamma_n}, (n,k)\in I \right)$ and consider so small $\omega>0$ that
$\ds \frac 83 \omega <  c_0$.
We compute from \eqref{dota} and \eqref{eq:BS}, for any $t\in [S_\star,S_0]$,
\begin{align*}
\frac d{dt} \mathcal A(\mathfrak a^+,t) 
& =  \frac 23 (p_0+2) \omega  \mathcal A(\mathfrak a^+,t)
+ 2 e^{\frac 23 (p_0+2) \omega t} \sum_{(n,k)\in I} \dot a_{n,k}^+(t) a_{n,k}^+(t)\\
& =  \frac 23 (p_0+2) \omega  \mathcal A(\mathfrak a^+,t)
- 2 e^{\frac 23 (p_0+2) \omega t} \sum_{(n,k)\in I}\frac{\lambda_{n,k}}{\gamma_n} |a^{\pm}_{n,k}(t)|^2
\\ & \quad +O\left(e^{\frac 23 (p_0+2) \omega t}e^{-\frac{1}{2} (p_0+1) \omega t}e^{-\frac 13 (p_0+2) t}\right)\\
& \le -c_0  \mathcal A(\mathfrak a^+,t) +O\left(e^{-\frac 1 6 (p_0-1) \omega t}\right).
\end{align*}
(because for all $(n,k) \in I$, $\ds 2\frac{\lambda_{n,k}}{\gamma_n}  \ge 2 c_0$ and $\ds \frac{2}{3}(p_0+2) \omega \le \frac{8}{3} \omega \le c_0$).
Now, at time $t=S_\star$, by \eqref{saturation}, we have $\mathcal A(\mathfrak a^+,S_\star)=1$ and so
\begin{align*}
\frac d{dt} \mathcal A(\mathfrak a^+,t)|_{t=S_\star(\mathfrak a)}
& \le -c_0    +O\left(e^{-\frac 1 6 (p_0-1) \omega t}\right) \le -\frac {c_0}2 <0,
\end{align*}
for $T_0$ large enough.
\end{proof}

\subsection{Conclusion: topological argument}

To conclude the proof of Proposition~\ref{pr:3}, we argue by contradiction, and assume that for any $\mathfrak a^+ \in \m B$,  \eqref{contra} holds, that is $S_\star(\mathfrak a^+) > T_0$. 

Then Lemma \ref{le:tra} applies to any $\mathfrak a^+ \in \m B$: as a standard consequence of this transversality result, the following map
\[
\mathcal M : \m B \to \m S, \quad \mathfrak a^+ \mapsto \mathcal M(\mathfrak a^+) := 
e^{\frac 13 (p_0+2) \omega S_\star} a^+(S_\star)
\]
is well defined on the unit ball $\m B$ with values in the unit sphere $\m S$ of $\R^{|I|}$, continuous, and its restriction to $\m S$ is the identity. A contradiction is reached from Brouwer's theorem. Hence there exists at least one $\mathfrak a^+ \in \m B$ such that $S_\star(\mathfrak a^+)= T_0$, and it provides the sought for solution $U$ of \eqref{NLKGs}.

We refer to \cite{CMM} and \cite{CoMu} for more details.

\bigskip
\bigskip
\bigskip

\begin{center}
{\scshape Rapha\"el C\^{o}te}
\smallskip\\
{\footnotesize
Université de Strasbourg\\
CNRS, IRMA UMR 7501\\
F-67000 Strasbourg, France\\
\email{cote@math.unistra.fr}
}
\end{center}

\bigskip

\begin{center}
{\scshape Yvan Martel}
\smallskip\\
{\footnotesize
% please put the address of the first author
CMLS, \'Ecole polytechnique, CNRS\\
91128 Palaiseau Cedex, France\\
\email{yvan.martel@polytechnique.edu}}
\end{center}
 % Do not forget to end the {\footnotesize by the sign }

\end{document}